\documentclass[12pt]{amsart}

\usepackage{geometry}
\usepackage{lmodern}
\usepackage[stretch=10]{microtype}
\usepackage{graphicx}
\usepackage{tikz-cd}

\usepackage{amsmath}
\usepackage{amssymb}
\usepackage{amsthm}
\usepackage[matrix,arrow,cmtip]{xy}
\usepackage{bbm}
\usepackage[shortlabels]{enumitem}
\PassOptionsToPackage{hyphens}{url}
\usepackage[bookmarksnumbered,bookmarksdepth=2,bookmarksopen,colorlinks,linkcolor=black,citecolor=black,urlcolor=black,linktoc=all]{hyperref}
\usepackage[capitalise]{cleveref}

\newcommand{\bA}{\mathbb{A}}
\newcommand{\bC}{\mathbb{C}}

\newcommand{\bG}{\mathbb{G}}

\newcommand{\bN}{\mathbb{N}}
\newcommand{\bP}{\mathbb{P}}
\newcommand{\bQ}{\mathbb{Q}}
\newcommand{\bR}{\mathbb{R}}

\newcommand{\bZ}{\mathbb{Z}}

\newcommand{\cA}{\mathcal{A}}

\newcommand{\cC}{\mathcal{C}}
\newcommand{\cD}{\mathcal{D}}
\newcommand{\cE}{\mathcal{E}}

\newcommand{\cG}{\mathcal{G}}

\newcommand{\cO}{\mathcal{O}}

\newcommand{\cS}{\mathcal{S}}

\newcommand{\cX}{\mathcal{X}}

\newcommand{\fm}{\mathfrak{m}}
\newcommand{\fp}{\mathfrak{p}}
\newcommand{\fs}{\mathfrak{s}}
\newcommand{\fz}{\mathfrak{z}}

\newcommand{\fC}{\mathfrak{C}}

\newcommand{\fG}{\mathfrak{G}}

\newcommand{\fV}{\mathfrak{V}}

\newcommand{\gG}{\mathbf{G}}

\newcommand{\gSL}{\mathbf{SL}}

\newcommand{\rM}{\mathrm{M}}

\DeclareMathOperator{\Aut}{Aut}

\DeclareMathOperator{\End}{End}

\DeclareMathOperator{\red}{red}

\DeclareMathOperator{\SL}{SL}

\DeclareMathOperator{\Spec}{Spec}
\DeclareMathOperator{\Spf}{Spf}

\DeclareMathOperator{\tr}{tr}
\DeclareMathOperator{\trdeg}{trdeg}

\newcommand{\rig}{\mathrm{rig}}

\newcommand{\formal}{\mathrm{for}}

\newcommand{\abs}[1]{\lvert #1 \rvert}

\newcommand{\ov}{\overline}

\newcommand{\fullmatrix}[4]{\left( \begin{matrix} #1 & #2 \\ #3 & #4 \end{matrix} \right)}

\newcommand{\defterm}[1]{\textbf{#1}}

\newtheorem{lemma}{Lemma}[section]
\AddToHook{env/lemma/begin}{\crefalias{lemma}{lemma}}
\newtheorem{proposition}[lemma]{Proposition}
\AddToHook{env/proposition/begin}{\crefalias{lemma}{proposition}}
\newtheorem{theorem}[lemma]{Theorem}
\AddToHook{env/theorem/begin}{\crefalias{lemma}{theorem}}

\AddToHook{env/corollary/begin}{\crefalias{lemma}{corollary}}
\newtheorem{conjecture}[lemma]{Conjecture}
\AddToHook{env/conjecture/begin}{\crefalias{lemma}{conjecture}}
\Crefname{conjecture}{Conjecture}{Conjectures} 

\Crefname{claim}{Claim}{Claims}
\newtheorem*{lemma*}{Lemma}
\newtheorem*{proposition*}{Proposition}
\newtheorem*{theorem*}{Theorem}
\newtheorem*{corollary*}{Corollary}
\newtheorem*{claim*}{Claim}

\theoremstyle{definition}
\newtheorem*{definition}{Definition}
\newtheorem{remark}[lemma]{Remark}
\AddToHook{env/remark/begin}{\crefalias{lemma}{remark}}

\usepackage{ifthen}
\newcounter{constant}

\newcommand{\newC}[1]{%
   \ifthenelse{\equal{#1}{*}} {%
      \stepcounter{constant} c_{\theconstant}%
   } {%
      \refstepcounter{constant} c_{\theconstant} \label{C:#1}%
   }%
}
\newcommand{\refC}[1]{c_{\ref*{C:#1}}}

\usepackage{color}
\newcommand{\chris}{\textcolor{red}}

\newcommand{\van}{{v\textrm{-}\mathrm{an}}}
\newcommand{\wan}{{w\textrm{-}\mathrm{an}}}
\newcommand{\ian}{{\van\textrm{-}\mathrm{an}}}

\newcommand{\Qbar}{\overline{\bQ}}
\newcommand{\powerseries}[2]{#1 [\![ #2 ]\!]}

\title{Some new cases of Zilber--Pink in $Y(1)^3$}

\author{Christopher Daw, Martin Orr, and Georgios Papas}

\address{}
\email{}

\dedicatory{}

\subjclass[]{}

\begin{document}



\maketitle

\begin{abstract}
    We prove the Zilber--Pink conjecture for curves in $Y(1)^3$ that intersect a modular curve in the boundary. We also give an unconditional result for unlikely intersection points having few places of supersingular reduction where they are close to a fixed base point. Both results are proved using the G-functions method for unlikely intersections. 
\end{abstract}

\quad

\section{Introduction}

The aim of this paper is to prove new cases of the Zilber--Pink conjecture for~$Y(1)^3$. The statement of this conjecture is that {\it a geometrically irreducible algebraic curve~$C\subset Y(1)^3 \cong \bA^3$, not contained in a proper special subvariety, has at most finitely many intersection points with special curves in $Y(1)^3$} (see below for the definition of special subvarieties in~$Y(1)^n$). Habegger and Pila~\cite{HP12} have proved this conjecture for so-called \textit{asymmetric} curves. More recently, the first- and second-named authors~\cite{Y(1)} and the third-named author~\cite{Papas:Y(1)} have proved it under restrictions on how the closure of~$C$ in $(\bP^1)^3$ intersects the boundary $(\bP^1)^3 \setminus \bA^3$.

In this paper, we prove the conjecture when $C$ intersects a strongly special curve in the boundary (Theorem~\ref{teo:main}), in contrast to our previous results which require $C$ to intersect certain zero-dimensional subsets of the boundary.
We also prove a version of the conjecture involving no condition at the boundary, but which only gives finiteness for points satisfying a condition on their places of supersingular reduction (Theorem~\ref{teo:super}).

\quad

The Zilber--Pink conjecture is one of the central problems in arithmetic geometry today. It can be formulated in many different settings.
For (pure) Shimura varieties, it may be stated as follows.

\begin{conjecture}[Zilber--Pink]
Let $V$ be an irreducible subvariety of a Shimura variety $S$. Suppose that the intersection
of $V$ with the union of all special subvarieties of $S$ of codimension greater than $\dim(V)$ is Zariski dense in $V$. Then $V$ is contained in a proper special subvariety of $S$.
\end{conjecture}

The simplest Shimura variety is $Y(1)\cong\bA^1$---the moduli space of elliptic curves. Clearly, Zilber--Pink says nothing for $Y(1)$ itself. For $Y(1)^2$ the conjecture asserts that an irreducible (plane) curve containing infinitely many special (CM) points is special. This was proved by Andr\'e \cite{andre:AO} in 1998, and is prototypical of the general Andr\'e--Oort conjecture, which states that {\it if an irreducible subvariety of a Shimura variety contains a Zariski dense set of special points, then it is a special subvariety}. The Andr\'e--Oort conjecture has recently been proved in full, through the works of many authors, culminating in a paper of Pila--Shankar--Tsimerman, with an appendix by Esnault--Groechenig \cite{PST+}.

Thus, the obvious next step is Zilber--Pink in $Y(1)^3$. With Andr\'e--Oort decided, this reduces to the statement that an irreducible curve $C$ in $\bA^3$ not contained in any proper special subvariety intersects only finitely many special curves.  

If $C$ is not defined over $\Qbar$, then Zilber--Pink holds for $C$ by work of Pila \cite[Thm. 1.4]{pila:fermat}. Over $\Qbar$, the work of Habegger--Pila \cite[Thm. 2]{HP12} establishes that $C$ will only intersect finitely many special curves with a fixed coordinate. We recall that the special subvarieties of $Y(1)^n$ are defined by taking irreducible components of those subvarieties defined by finite conjunctions of
\begin{itemize}
\item[(1)] $X_i=x$, for some $x$ in the set $\Sigma_{\rm CM}$ of singular moduli (i.e.\ $j$-invariants of CM elliptic curves);
\item[(2)] $\Phi_N(X_i,X_j)=0$, for some $1\leq i<j\leq n$ and $N\in\mathbb{N}$,
\end{itemize}
where $\Phi_N(X,Y)\in\bZ[X,Y]$ denotes the $N$th modular polynomial.
This is defined by the property that two elliptic curves with $j$-invariants $j_1$ and $j_2$ are isogenous, with minimal isogeny of degree $N$, if and only if $\Phi_N(j_1,j_2)=0$. Special subvarieties defined by conditions only of type (2)---namely, those with no fixed coordinate---are examples of so-called {\it strongly special} subvarieties \cite{CU:equidistribution}.

It follows that the problem of Zilber--Pink in $Y(1)^3$ is reduced to the following, rather elegant conjecture. 

\begin{conjecture}\label{conj:ZPY(1)}
    Let $C$ be an irreducible algebraic curve in $Y(1)^3\cong\bA^3$ defined over $\Qbar$ not contained in a proper special subvariety. Then $C$ contains only finitely many points $(x_1,x_2,x_3)$ for which there exist $M,N\in\bN$ such that
    \[\Phi_M(x_1,x_2)=\Phi_N(x_2,x_3)=0.\]
\end{conjecture}

Because the André--Oort conjecture is known, it suffices to consider only those points $s=(x_1,x_2,x_3)\in Y(1)^3$ whose coordinates are non-singular moduli. For such points, the positive integers $M,N\in\bN$ for which
\begin{equation} \label{eqn:two-modular-relations}
\Phi_M(x_1,x_2)=\Phi_N(x_2,x_3)=0
\end{equation}
are uniquely determined.

We henceforth refer to a point $s=(x_1,x_2,x_3) \in (Y(1) \setminus \Sigma_{\rm CM}) ^3$ satisfying~\eqref{eqn:two-modular-relations} as {\it modular}.
For a modular point, we define 
\[\Delta(s):=\max\{M,N\}.\]

By an ingenious argument, Habegger and Pila showed that Conjecture \ref{conj:ZPY(1)} holds when $C$ is {\it asymmetric}, which (in this case) is to say that the degrees of the coordinate projections $C\to Y(1)$ are not all equal \cite[Thm.~1]{HP12}. In some sense, this covers most curves. On the other hand, it is not clear to us how to generalize their argument, or how to deal with symmetric curves specifically.

With this in mind, the present authors have more recently explored an alternative approach, based on techniques of Andr\'e and Bombieri using {\it G-functions}.     
In \cite{Y(1)}, the first- and second-named authors showed that Conjecture \ref{conj:ZPY(1)} holds for any $C$ whose Zariski closure $\overline{C}$ in $X(1)^3\cong(\bP^1)^3$ (the Baily--Borel compactification) contains $(\infty,\infty,\infty)$. In \cite{Papas:Y(1)}, the third-named author observed that this could be extended to any curve $C$ for which $\overline{C}\setminus C$ contains a point $(x_1,x_2,x_3)$ with $x_1,x_2,x_3\in\{\infty\}\cup\Sigma_{\rm CM}$. In other words, $\overline{C}$ intersects what we might call a {\it special point in the boundary}.

Note that the boundary $X(1)^3 \setminus Y(1)^3 = (\bP^1)^3 \setminus \bA^3$ can naturally be identified with the disjoint union of three copies of $Y(1)^2$ (where exactly one coordinate is~$\infty$), three copies of $Y(1)$ (where exactly two coordinates are~$\infty$), and the singleton $\{(\infty,\infty,\infty)\}$.

Since $\overline{C}$ must always intersect the boundary, it is interesting to ask which intersection conditions facilitate our method. A natural next step in this direction is the case of a curve intersecting a special curve in one of the copies of $Y(1)^2$ in the boundary. 

We say that a point in the boundary is {\it modular} if it belongs to one of the copies of $Y(1)^2$ and its coordinates $x_1,x_2$ satisfy $\Phi_N(x_1,x_2)=0$ for some $N\in\mathbb{N}$. In this terminology, we prove the following.

\begin{theorem}\label{teo:main}
    Conjecture \ref{conj:ZPY(1)} holds under the assumption that $\overline{C}\setminus C$ contains a modular point.
\end{theorem}

Our methods also yield a Zilber--Pink-type result without conditions at the boundary, but with an extra restriction on the modular points of~$C$ which we consider. In order to state this result, let $C$ be as in the statement of Conjecture \ref{conj:ZPY(1)}. Observe that we can always assume that $C$ contains a modular point $s_0=(b_1,b_2,b_3)$.
(Otherwise, there is nothing to prove!) Observe also that any modular point on $C$ belongs to $C(\Qbar)$. We denote by $K_0$ the number field $\bQ(b_1,b_2,b_3)$, and we say that a finite place $v$ of any finite extension of $K_0$ is a {\it place of supersingular reduction} if the elliptic curve with $j$-invariant $b_1$ (equivalently, $b_2$ or $b_3$) has supersingular reduction at $v$ (for the minimal Weierstrass model over $K_v$). For any such place~$v$, lying over a rational prime $p$, we denote by $|\cdot|_v$ the absolute value extending the standard $p$-adic absolute value on $\bQ$. 

\begin{definition}
Given $\delta > 0$, we say that a modular point $s=(x_1,x_2,x_3)\in C$ has {\bf supersingular exponent} $\delta$ if the set of places $v$ of $K_0(x_1,x_2,x_3)$ of supersingular reduction for which the inequalities
\begin{equation} \label{eqn:supersing-close-condition}
    |x_1-b_1|_v,|x_2-b_2|_v,|x_3-b_3|_v<1
\end{equation}
hold has cardinality at most $\Delta(s)^{\delta}$. 
\end{definition}

\begin{remark}
This definition may be read as saying that $s$ is $v$-adically close to~$s_0$ for only a small number of places of supersingular reduction. Note that, if $v$ is a place of supersingular reduction satisfying~\eqref{eqn:supersing-close-condition}, then the elliptic curves with $j$-invariants $x_1$, $x_2$, $x_3$ also have supersingular reduction at~$v$.
\end{remark}

\begin{remark}
By a result of Serre, the set of places of supersingular reduction has density zero. On the other hand, it is conjectured to be infinite, and this is known if $K_0$ has at least one real embedding \cite{Elkies} (see \cite{Charles} for a related result). However, for any $K_0$ and any $s \neq s_0$, there are only finitely many places satisfying~\eqref{eqn:supersing-close-condition}.
\end{remark}

In this terminology, we prove the following.

\begin{theorem}\label{teo:super}
    Let $C$ be as in \cref{conj:ZPY(1)}. Let $(b_1,b_2,b_3)$ be a modular point in~$C$. There exists $\delta>0$ such that $C$ contains only finitely many modular points of supersingular exponent $\delta$, with respect to~$(b_1,b_2,b_3)$.
\end{theorem}

The constant $\delta$ afforded to us by Theorem \ref{teo:super} is relatively explicit. Indeed, any $\delta$ satisfying the conditions of Theorem \ref{teo:galois-super} is allowed. By Remark \ref{rem:constant}, therefore, we are allowed to take any $\delta<(72\ell)^{-1}$, where $\ell$ (depending on~$C$) is defined in Section~\ref{sec:cover}. It is straightforward to control $\ell$ in terms of the genus of $C$.

To prove Theorems \ref{teo:main} and \ref{teo:super}, we obtain the following height bounds. The derivations of Theorems \ref{teo:main} and \ref{teo:super} from these height bounds is now standard---using the so-called Pila--Zannier strategy and Masser--W\"ustholz-type isogeny estimates---but we will give a few details in Section \ref{sec:derivations}. In order to state our height bounds, we denote by $h(x_1,x_2,x_3)$ the maximum of the {\it absolute logarithmic heights} of the $x_i$, as defined, for example, in \cite[\S 2.2]{HP12}.

\begin{theorem}\label{teo:main-ht}
    Let $C$ satisfy the conditions of Theorem \ref{teo:main}. There exist constants $\newC{main-ht-mult}$ and $\newC{main-ht-exp1}$ such that, for any modular point $(x_1,x_2,x_3)\in C$, we have
    \[h(x_1,x_2,x_3)\leq\refC{main-ht-mult}[\bQ(x_1,x_2,x_3):\bQ]^{\refC{main-ht-exp1}}.\]
\end{theorem}

\begin{theorem}\label{teo:super-ht}
    Let $C$ and $(b_1,b_2,b_3)$ be as in Theorem \ref{teo:super}. There exist constants $\newC{super-ht-mult}$, $\newC{super-ht-exp}$ and $\newC{super-ht-exp2}$ such that, for any modular point $s=(x_1,x_2,x_3)\in C$ of supersingular exponent $\delta$, we have
    \[h(x_1,x_2,x_3)\leq\refC{super-ht-mult}[\bQ(x_1,x_2,x_3):\bQ]^{\refC{super-ht-exp}}\Delta(s)^{\refC{super-ht-exp2}\delta}.\]
\end{theorem}

As mentioned above, Theorems \ref{teo:main-ht} and \ref{teo:super-ht} are obtained via the so-called {\it G-function method}. We refer to \cite[Section 1.D]{LGO} for a detailed history of this method and its bearing on problems of unlikely intersections. The original motivation, and one of the key novelties, of this paper was to carry out the method around a base point in the interior of a Shimura variety, as opposed to a point of degeneration at the boundary. Theorem \ref{teo:super}, which is close in spirit to \cite[Th\'eor\`eme 1]{And95} (and whose proof is very much inspired by the ideas introduced in \cite{And89} and \cite{And95}) is our main result in this direction. This line of inquiry is followed up in a series of recent papers of the third-named author \cite{PapasI,PapasII,PapasIII}.

\subsection*{Ideas of the proof}

Consider a geometrically irreducible algebraic curve $C \subset Y(1)^3$ defined over a number field~$K$, and a modular point $s_0$, either in $C(K)$ or in~$\ov C(K) \setminus C(K)$.
Let $\cE_1, \cE_2, \cE_3$ denote families of elliptic curves over a suitable Zariski open subset of~$C$ whose $j$-invariants are given by the three projections $C \to Y(1)$.

The Taylor expansions around~$s_0$ of suitable linear combinations of complex periods of these familes of elliptic curves are G-functions (Section~\ref{sec:GF-ab-sch}, following \cite[Sec.~1]{And95}).
We need to interpret the evaluations of these G-functions at non-archimedean places in terms of periods which are suitably functorial, by choosing a suitable notion of ``horizontal'' sections of de Rham cohomology over an analytic neighbourhood~$s_0$.
At places of good reduction for~$\cE_{1,s_0}$, this is achieved using Berthelot and Ogus's comparison between crystalline and de Rham cohomology \cite{BO1983,Ogus84}; in fact, the construction for non-archimedean places of good reduction closely parallels that for archimedean places, so that we are able to treat both in a uniform way in Section~\ref{sec:Case1}.
At places of bad reduction, we instead use Tate uniformisation, similarly to the case of families with multiplicative degeneration at a $K$-point dealt with in~\cite{Y(1),LGO} (Section~\ref{sec:Case2}).

Let $s\in C(L)$ be a modular point, for some finite extension $L$ of~$K$.
For each place~$v$ of~$L$, we use linear algebra calculations starting from the existence of an isogeny $\cE_{1,s} \to \cE_{2,s}$ to construct a non-trivial $v$-adic relation between the evaluations of our G-functions at~$s$ (Sections \ref{sec:Case1} and~\ref{sec:Case2}).
In order to obtain a global relation, we need to multiply these $v$-adic relations together.
However, in general the relations at different places are different, and we do not know how many of them we will need to multiply together, so we do not have control of the degree of the global relation, and hence cannot prove Conjecture~\ref{conj:ZPY(1)} in full.

This is where the extra conditions of Theorems \ref{teo:main-ht} and~\ref{teo:super-ht} come in.
In Theorem~\ref{teo:main-ht} (where $\cE_3$ degenerates to~$\bG_m$ at~$s_0$), a similar trick to that in \cite[Ch. X, 3.1 and~3.4]{And89} ensures that only archimedean places and places of bad reduction are relevant; the number of each of these is proportional to~$[L:\mathbb{Q}]$, so this gives the bound we need (Section~\ref{proof:main-ht}, especially Section~\ref{sec:degree}).
In Theorem~\ref{teo:super-ht}, by taking into account both isogenies $\cE_{1,s} \to \cE_{2,s}$ and $\cE_{1,s} \to \cE_{3,s}$, we are able to construct a single relation of fixed degree which is valid at all relevant places of ordinary good reduction.
We are not able to do the same thing at places of supersingular good reduction, so instead we use the notion of supersingular exponent to impose a restriction on the number of places of supersingular reduction as a condition of the theorem (Sections \ref{sec:triples} and~\ref{proof:super-ht}).

\subsection*{Structure of the paper}

In Section \ref{sec:prelim}, we give standard constructions pertaining to the G-function method. In Section \ref{sec:main}, we explain how to construct $v$-adic relations between the G-functions associated with elliptic schemes having isogenous fibres at a fixed point. In Section \ref{sec:proofs}, we prove Theorems \ref{teo:main-ht} and \ref{teo:super-ht}. In Section \ref{sec:derivations}, we sketch the (now standard) arguments for deducing Theorems \ref{teo:main} and \ref{teo:super} from Theorems \ref{teo:main-ht} and \ref{teo:super-ht}, respectively. In the appendix, we prove some results on rigid analytic neighbourhoods to be used in this and future works. These statements are surely well-known, but we found it convenient to put them together in one place. 

\subsection*{Acknowledgements}

The first-named author is indebted to Gregorio Baldi, with whom he had several interesting discussions around the idea of intersections at the boundary, and to Ananth Shankar, with whom he discussed several aspects of this work. He acknowledges the use of ChatGPT for conceptualisation. For part of this work, he was supported by a grant from the School of Mathematics at the Institute for Advanced Studies, Princeton.

The second-named author's work on this project was supported by the Engineering and Physical Sciences Research Council (EP/Y020758/1).

The third-named author started work on this project while supported by Michael Temkin's ERC Consolidator Grant BirNonArchGeom (770922). For the majority of the work, he received funding from the European Union (ERC, SharpOS, 101087910) and the ISRAEL SCIENCE FOUNDATION (grant No. 2067/23). Latterly, he was supported by the Minerva Research Foundation Member Fund, while in residence at the Institute for Advanced Study, for the academic year 2025-26. 

The authors thank Jacob Tsimerman for kindly answering questions related to his work \cite{BT:GPC} and for several other enlightening discussions.

\section{Preliminaries}\label{sec:prelim}

In this section, we assemble various notation and background material, which will be used throughout the rest of the paper.

\subsection{Matrices} 

For a square matrix $M$, we denote its trace by $\tr(M)$ and its determinant by $\det(M)$. We denote its adjugate by $M^{\rm adj}$ (which, if $M$ is invertible, is equal to $\det(M)M^{-1}$) and its transpose by $M^T$. For any ring $R$ and positive integer~$n$, we denote by $\rM_n(R)$ the ring of $n\times n$ matrices with coefficients in $R$.

\subsection{Fields}

For a number field $K$, we denote its ring of integers by $\cO_K$ and its algebraic closure by $\overline{K}$. By a place $v$ of $K$, we refer to the equivalence class of an absolute value on $K$. We denote the set of places of $K$ by $\Sigma_K$, and we denote the set of archimedean (resp.\ finite, or non-archimedean) places of $K$ by $\Sigma_{K,\infty}$ (resp.\ $\Sigma_{K,f}$). For $v\in\Sigma_{K,f}$, we will use the following notations:  
\begin{itemize}
    \item[] $K_v:=$ the completion of $K$ with respect to $v$
    \item[] $\cO_{K_v}:=$ the ring of integers in $K_v$
    \item[] $\fp_v:=$ the prime ideal of $\cO_K$ corresponding to $v$
    \item[] $k_v:=$ the residue field $\cO_K/\fp_v$
    \item[] $W(k_v):=$ the ring of Witt vectors over $k_v$
\end{itemize}
By abuse of notation, for $v\in\Sigma_{K,\infty}$, we define  $K_v:=\bC$ (even if the completion of~$K$ with respect to~$v$ is~$\bR$). 

For $v\in\Sigma_K$, we denote by $|\cdot|_v$ the absolute value on $K_v$ extending the standard absolute value on $\bQ$. We write $(\cdot)^\van$ for the analytification functor from schemes locally of finite type over $K_v$ to (rigid or complex) analytic spaces over $K_v$. For a power series $F\in K[\![X]\!]$, we let $F^\van$ denote the corresponding analytic function defined on the open disc $D(0,R(F^\van),K_v)$, where $R(F^\van)$ is the $v$-adic radius of convergence.


\subsection{Elliptic curves}

For a scheme $X\to\Spec R$ and an $R$-algebra $R\to S$, we denote by $X_S$ the base-change $X\times_{\Spec R}\Spec S$. For an elliptic curve $E$ over $K$, we say that $E$ has good reduction at $v\in\Sigma_{K,f}$ if the special fibre $\cE_{k_v}$ of the minimal Weierstrass model $\cE$ of $E_{K_v}$ over~$\cO_{K_v}$ is non-singular. We say $E$ has bad reduction at $v$ otherwise. If $\cE_{k_v}$ is non-singular, it is an elliptic curve over $k_v$, and we say that $E$ has ordinary (resp.\ supersingular) reduction at $v$ if $\cE_{k_v}$ is ordinary (resp.\ supersingular). All of these properties are preserved by isogenies over $K$.

Now let $\cX$ be a semiabelian scheme over $\Spec\cO_K$ satisfying $\cX_K\cong E$. Then, for any $v\in\Sigma_{K,f}$, the base-change $\cX_{\cO_{K_v}}$ is the connected N\'eron model of $E_{K_v}$ over $\Spec\cO_{K_v}$ \cite[Prop.~7.4/3]{BLR90}. It follows that $\cX_{k_v}$ is an elliptic curve if and only if $E$ has good reduction at $v$. Moreover, if $\cX_{k_v}$ is an elliptic curve, it is ordinary (resp.\ supersingular) if and only if $E$ has ordinary (resp.\ supersingular) reduction at $v$.


\subsection{Taylor series with respect to a local parameter}\label{sec:TS}
Let $C$ be a smooth geometrically irreducible algebraic curve defined over a number field $K$. Let $s_0\in C(K)$ and let $x\in K(C)$ be a local parameter at $s_0$ (which is to say that $x$ is a generator for the maximal ideal of the local ring $\cO_{C,s_0}$). Let $v\in\Sigma_K$ and let $\fm$ denote the maximal ideal of the stalk $\cO_{C,v,s_0}$ of $\cO_{C^\van}$ at $s^{\van}_0$ (which is generated by the equivalence class of $x^{\van}$). Since $\cO_{C,v,s_0}$ is Noetherian, it embeds into its $\fm$-adic completion $\hat\cO_{C,v,s_0}$. Furthermore, since $\hat\cO_{C,v,s_0}$ is regular, it is isomorphic to $K_v[\![X]\!]$, and we may choose the isomorphism so that it sends the germ of~$x^\van$ to the indeterminate~$X$. Let $T$ denote the resulting injective ring homomorphism
\[T:\cO_{C,v,s_0}\to K_v[\![X]\!]. \]
For any $f\in \cO_{C,v,s_0}$, we refer to $T(f)$ as {\bf the Taylor series of $f$ with respect to $x$}.

\subsection{G-functions}

A \defterm{G-function} is a power series $F(X) = \sum_{n \geq 0} a_nX^n$ with coefficients $a_n$ in a number field $K$ which satisfies the following conditions:
\begin{enumerate}
\item there exists $\newC{G-function-arch-base}$ such that $\abs{a_n}_v < \refC{G-function-arch-base}^n$ for all $n \geq 1$ and all $v\in\Sigma_{K,\infty}$;
\item there exists a sequence of positive integers $(d_n)$, which grows at most geometrically, such that $d_na_m$ is an algebraic integer for all $m \in \{1, \dotsc, n\}$;
\item $F$ satisfies a linear homogeneous differential equation
\[ \frac{d^\mu}{dX^\mu}F + \gamma_{\mu-1} \frac{d^{\mu-1}}{dX^{\mu-1}}F + \dotsb + \gamma_1 \frac{d}{dX}F + \gamma_0F = 0 \]
with coefficients $\gamma_i \in K(X)$.
\end{enumerate}


\subsection{G-functions associated with an abelian scheme}\label{sec:GF-ab-sch}

Let $K$, $C$, $s_0$, $x$ be as in Section \ref{sec:TS} and let $\pi \colon \cA\to C$ be an abelian scheme of relative dimension $g$, also defined over $K$. We write $H^1_{DR}(\cA/C)$ for the relative algebraic de Rham cohomology, equipped with its Gauss--Manin connection 
\[\nabla:H^1_{DR}(\cA/C)\to H^1_{DR}(\cA/C)\otimes_{\cO_C}\Omega^1_C.\]

By \cite[Prop.~6.14]{Del70}, the $\cO_C$-sheaf $H^1_{DR}(\cA/C)$ is locally-free. Suppose that it is free (which we will later achieve by replacing $C$ with a Zariski open subset containing~$s_0$). Equivalently, we can choose a basis $\{\omega_1,\ldots,\omega_{2g}\}$ of global sections of $H^1_{DR}(\cA/C)$.

\begin{lemma} \label{lem:horizontal-sections}
Let $v \in \Sigma_K$.
There exists an open neighbourhood $\Delta_v$ of~$s_0^\van$ such that $H^1_{DR}(\cA/C)^\van(\Delta_v)$ possesses a basis of horizontal sections $\{\gamma_{v,1},\ldots,\gamma_{v,2g}\}$.
\end{lemma}

\begin{proof}
For $v\in\Sigma_{K,\infty}$, the existence of such a~$\Delta_v$ is classical and follows from Cauchy's theorem---any simply connected open subset would suffice. For $v\in\Sigma_{K,f}$, the existence of such a~$\Delta_v$ is guaranteed by \cite[Thm.~IV]{Lutz}.
\end{proof}

Let $v \in \Sigma_K$.
Let $\Delta_v$ and $\{\gamma_{v,1},\ldots,\gamma_{v,2g}\}$ be as in \cref{lem:horizontal-sections}.
Let 
\[\Omega_v\in\rM_{2g}(\cO_{C^\van}(\Delta_v))\] be the matrix whose rows give the coordinates of the $\omega^{\van}_i$ in terms of the $\gamma_{v,j}$.

\begin{remark}\label{rem:periods}
In the case $v \in \Sigma_{K,\infty}$, we may choose the horizontal sections $\gamma_{v,1}, \dotsc, \gamma_{v,2g}$ so that they correspond to sections of the constant local system $R^1\pi^\van_*\Qbar$ under the comparison isomorphism between de Rham and Betti cohomology.
Then $\Omega_v$ is a matrix of complex periods of the family of abelian varieties $\cA \to C$.
\end{remark}

\begin{lemma}\label{lem:GF-ind-v}
Define
\[ Y_v:=\Omega_v\cdot\Omega_v(s^\van_0)^{-1}\in\rM_{2g}(\cO_{C^\van}(\Delta_v)) \]
and denote its entries by $Y_{v,mn}$ (where $1\leq m,n\leq 2g$).
Let $F_{v,mn}$ denote the Taylor series of $Y_{v,mn}$ with respect to $x$. Then the $F_{v,mn}$ are G-functions, belonging to $K[\![X]\!]$, and are independent of $v\in\Sigma_K$. 
\end{lemma}

Note that $Y_v$ and the $F_{v,mn}$ are independent of the choice of $\gamma_{v,1},\dotsc,\gamma_{v,2g}$. We make use of this in Section \ref{proof:rel}, where we impose additional conditions on this basis and rely on the fact that this leaves the $F_{v,mn}$ unaffected.

\begin{proof}
It is straightforward to verify that the matrix $F_v:=(F_{v,mn})_{mn}$ satisfies a linear differential equation $\frac{d}{dx}(-)=A(-)$ induced from the action of $\nabla(\frac{d}{dx})$ on $H^1_{DR}(\cA/C)$, where $A\in\rM_{2g}(K((X)))$. Starting from the fact that $F_v(0)=I_{2g}$ has entries in $K$, it follows that $F_{v,nm}\in K[\![X]\!]$ for all $m,n\in\{1,\ldots,2g\}$. By \cite[p.~3, Thm.~B]{And89}, therefore, $F_{v,nm}$ is a G-function for all $m,n\in\{1,\ldots,2g\}$.
Furthermore, $F_v$ is the unique solution of $\frac{d}{dx}(-)=A(-)$ in $\rM_{2g}(K((X)))$ satisfying $F_v(0)=I_{2g}$, so it is independent of~$v$.
\end{proof}

Because the $F_{v,mn}$ in Lemma \ref{lem:GF-ind-v} are independent of~$v$, we denote them simply by~$F_{mn}$.
We refer to the $F_{mn}$ for $1\leq m,n\leq 2g$ as {\bf the G-functions associated with $\cA\to C$, $x$, $s_0$ and $\{\omega_1,\ldots,\omega_{2g}\}$}.





\subsection{Relations}

Let $F_1, \dotsc, F_n \in \Qbar[\![X]\!]$. We say that a polynomial $\tilde Q\in \Qbar[X][X_1, \dotsc, X_n]$ is a \defterm{functional relation} between $F_1, \dotsc, F_n$ if it is homogeneous in $X_1,\dotsc,X_n$ and
\[\tilde Q(X)(F_1(X), \dotsc, F_n(X)) = 0\text{ in }\Qbar[\![X]\!].\]

Let $\xi \in \Qbar$ and let $Q \in \Qbar[X_1, \dotsc, X_n]$ be a homogeneous polynomial.
Let $K$ be a number field which contains~$\xi$ and all of the coefficients of $F_1, \dotsc, F_n$ and~$Q$. Following \cite[Ch. VII, 5.1]{And89}, we make the following definitions:
\begin{enumerate}
\item $Q$ is an \defterm{$v$-adic relation} between the evaluations at $\xi$ of $F_1, \dotsc, F_n$ (for $v\in\Sigma_K$) if $\abs{\xi}_v < \min \{ 1, R(F_1^\van), \dotsc, R(F_n^\van) \}$ and
\[ Q^\van(F_1^\van(\xi^\van), \dotsc, F_n^\van(\xi^\van)) = 0. \]
\item $Q$ is a \defterm{global relation} between the evaluations at $\xi$ of $F_1, \dotsc, F_n$ if it is a $v$-adic relation between the evaluations at~$\xi$ for every place $v$ of~$K$ satisfying $\abs{\xi}_v < \min \{ 1, R(F_1^\van), \dotsc, R(F_n^\van) \}$.
\item $Q$ is a \defterm{trivial relation} between $F_1, \dotsc, F_n$ at~$\xi$ if it is the specialisation at $X=\xi$ of a functional relation $\tilde Q \in \Qbar[X][X_1, \dotsc, X_n]$ between $F_1, \dotsc, F_n$.
\end{enumerate}

\subsection{Simple neighbourhoods}\label{sec:simp-n}

Return to the situation of Section \ref{sec:TS} and let $\fC$ be a regular integral flat $\cO_K$-scheme with an isomorphism $\fC_K \to C$. Let $v\in \Sigma_{K,f}$ and let $\hat{\fC}_{\fC_{k_v}}$ denote the formal completion of~$\fC$ along the fibre $\fC_{k_v}$. Let $\hat{\fC}_{\fC_{k_v}}^\rig$ denote the associated rigid analytic space (see \cite[0.2.6]{Berthelot96}). Then there is a canonical open immersion $\hat{\fC}_{\fC_{k_v}}^\rig\to C^\van$ \cite[Prop. 0.3.5]{Berthelot96} and reduction map $\red_v:\hat{\fC}_{\fC_{k_v}}^\rig \to \fC_{k_v}$.
We say that an open subspace $U$ of $C^\van$ is \defterm{simple} (with respect to the integral model~$\fC$) if all of the following hold:
\begin{enumerate}
    \item $U \subset \hat{\fC}_{\fC_{k_v}}^\rig$;
    \item ${\red_v}|_U$ is constant;
    \item the value of $\red_v$ on~$U$ is a smooth closed point~$t\in\fC_{k_v}$.
\end{enumerate}
(In other words, $U$ is contained in the residue disc of a smooth closed point of~$\fC_{k_v}$.)

\section{Relations for elliptic schemes}\label{sec:main}

Let $\cE_1, \cE_2$ be elliptic schemes over the same curve~$C$ having isogenous fibres at some base point~$s_0$.
In this section, we construct $v$-adic relations between the G-functions associated with~$\cE_1,\cE_2$ (over the same curve~$C$) at points $s \in C(\Qbar)$ for which the fibres are again isogenous. In Sections \ref{sec:pairs} and~\ref{proof:rel}, we consider a pair of elliptic schemes, as just described, and construct relations which are potentially different at each place~$v$. In Section \ref{sec:triples}, we introduce a third elliptic scheme, in order to obtain a relation independent of~$v$, valid at all places~$v$ of good ordinary reduction of the fibres over the base point.

\subsection{Isogenous pairs}\label{sec:pairs} Let $C$ be a smooth projective geometrically irreducible algebraic curve, defined over a number field $K$, and let $\fC$ be a regular $\cO_K$-model of $C$ (see \cite[Definition 10.1.1]{Liu06} for the definition; such a model exists by \cite[Cor.~8.3.45]{Liu06}). 

\begin{definition}
    For any finite extension $L$ of $K$ and $s\in C(L)$, we refer to the unique section $\fs \colon \Spec(\cO_{L}) \to \fC_{\cO_{L}}$ whose image is the Zariski closure of~$s$ as {\bf the $\cO_{L}$-extension of $s$} (see \cite[Lemma~6.3]{LGO} for the existence of such a section). Note that, for any $v\in\Sigma_{K,f}$, the point $\fs(\fp_v)\in\fC_{k_v}$ is smooth (see \cite[Exercise 4.3.25(c)]{Liu06}).
\end{definition}

For $i\in\{1,2\}$, let $\fG_i\to\fC$ be a semiabelian scheme such that, for some non-empty Zariski open subset $C_0\subset C$, the base-change $\cE_i:=\fG_i\times_{\fC}C_0$ is an elliptic scheme over~$C_0$. Let $\fC_0\subset\fC$ denote a non-empty Zariski open subscheme such that $\fG_{i,0}:=\fG_i\times_{\fC}\fC_0$ is an elliptic scheme over $\fC$ for $i\in\{1,2\}$ (such a subscheme exists because the toric rank is upper semi-continuous---see \cite[Lemma~3.3.1.4]{Lan13}).
Suppose that $\fC_{0,K}=C_0$.


Let $s_0\in C_0(K)$ and suppose that there exists a $K$-isogeny $f_0:\cE_{1,s_0}\to\cE_{2,s_0}$ of degree $N_0:=\deg(f_0)$. We will say that $v\in\Sigma_{K,f}$ is a place of good (resp.\ bad) reduction if one (equivalently, both) of the $\cE_{i,s_0}$ have good (resp.\ bad) reduction at $v$.

Choose bases $\{\omega_i,\eta_i\}$ for $H^1_{DR}(\cE_i/C_0)$ such that
\begin{equation} \label{eqn:omega-f0}
\omega_{1,s_0}=\mu f^*_0\omega_{2,s_0} \text{ and } \eta_{1,s_0}=\mu f^*_0\eta_{2,s_0},
\end{equation}
for some $\mu\in\bQ$, where 
\[f^*_0:H^1_{DR}(\cE_{2,s_0}/K)\to H^1_{DR}(\cE_{1,s_0}/K)\]
denotes the pullback induced by $f_0$ ($\mu$ will afford us a degree of freedom which we will use in Section \ref{sec:proofs}). Let $x\in K(C_0)$ be a local parameter at $s_0$ and let $F_{imn}\in K[\![X]\!]$ denote the G-functions associated with $\cE_i\to C_0$, $x$, $s_0$ and $\{\omega_i,\eta_i\}$. Let 
\[\cG:=\{F_{imn}:1\leq i,m,n\leq 2\}.\]

For each $v \in \Sigma_{K,\infty}$, let $r_v$ and $U_v$ denote, respectively, a positive real number and an open neighbourhood of $s_0^\van$ in~$C^\van_0$ such that $x^\van$ restricts to an isomorphism of $K_v$-analytic spaces from $U_v$ to $D(0,r_v,K_v)$.
For the $v \in \Sigma_{K,f}$, let $r_v$ and~$U_v$ be the positive real numbers and simple open subspaces of $C^\van_0$ afforded to us by Proposition~\ref{rigid-coordinate-disc}. Define
\[R_v:=\min\{1,r_v,R(F^\van):F\in\cG\}\leq r_v\]
and $\Delta_v:=(x^\van|_{U_v})^{-1}(D(0,R_v,K_v))$. 


Note that we have $R_v=1$ for almost all $v\in\Sigma_{K}$. Moreover, $\Delta_v$ is simply connected for $v\in\Sigma_{K,\infty}$ and simple for $v\in\Sigma_{v,f}$.

\begin{lemma}\label{lem:equal-reduction}

    Let $s\in C(L)$ for some finite extension $L$ of $K$. Let $v\in\Sigma_{K,f}$ be a place of good reduction and let $w$ be a finite place of $L$ lying above $v$ such that $s^{\wan}\in\Delta_v$. Then, if $\fs,\fs_0:\Spec\cO_{L}\to\fC_{\cO_{L}}$ denote the $\cO_{L}$-extensions of $s$ and $s_0$, respectively, there are canonical isomoprhisms
    \[\fG_{1,\fs,k_{w}}\cong\fG_{1,\fs_0,k_{w}}\text{ and }\fG_{2,\fs,k_{w}}\cong\fG_{2,\fs_0,k_{w}}.\]  
\end{lemma}

\begin{proof}
    Since $\Delta_v$ is simple, we have $\red_{w}(s^\wan)=\red_{w}(s^\wan_0)$. Equivalently, $\fs(\fp_{w})=\fs_0(\fp_{w})$. The isomorphisms now follow by definition. 
\end{proof}

    We refer to the above data, namely,
    \[\{C,\fC,\fG_{i},C_0,\cE_{i},\fC_0,s_0,f_{0},\{\omega_{i},\eta_{i}\},x,F_{imn},R_v,\Delta_v:1\leq i,m,n\leq 2, v\in\Sigma_K\}\]
    as an {\bf isogeny datum}. The main result of this section is as follows.

\begin{proposition}\label{prop:rel} Consider an isogeny datum as above.

    Let $L$ be a finite extension of $K$ and suppose $s\in C_0(L)$ is such that there exists an $L$-isogeny $f:\cE_{1,s}\to\cE_{2,s}$.
    
    Let $v\in\Sigma_{K}$ and choose $w\in\Sigma_{L}$ lying over $v$. 
    Suppose that $s^\wan \in \Delta_v$.
    
    There exists a $w$-adic relation $Q_{w}$ of degree at most $2$ between the evaluations of the elements of $\cG$ at $x(s)$ which is not contained in the ideal $I$ of $\Qbar[X_{i mn}:1\leq i,m,n\leq 2]$ generated by the elements 
    \[\det((X_{1 mn})_{mn})-\det((X_{2 mn})_{mn}).\] 
    \end{proposition}


\subsection{Proof of Proposition \ref{prop:rel}}\label{proof:rel}


If necessary, replace $K$ with $L$ and $v$ with $w$ and, for uniformity of notation, define the field
\begin{align*}
    F_v := 
     \begin{cases}
       \bQ &\quad\text{if }v\in\Sigma_{K,\infty}\\
       K_v &\quad\text{if }v\in\Sigma_{K,f}.
     \end{cases}
\end{align*}
(This is the field of coefficients for the ``$v$-adic cohomology'' theory, namely, Betti cohomology when $v\in \Sigma_{K,\infty}$ or crystalline cohomology when $v \in \Sigma_{K,f}$.)
We will split the proof of Proposition \ref{proof:rel} into two cases.

\subsubsection{Case 1}\label{sec:Case1}
Suppose $v\in\Sigma_{K,\infty}$ or that $v\in\Sigma_{K,f}$ and the $\cE_{i,s_0}$ have good reduction at~$v$ (which is to say that the  $\cE_{i,s_0,k_v}:=\fG_{i,\fs_0,k_v}$ are elliptic curves). 

\quad

\paragraph{\bf Setup}
Define the $F_v$-vector space
\begin{align*}
    H^1_v(\cE_{i,s}) := 
     \begin{cases}
       H^1(\cE_{i,s}^\van,F_v)&\quad\text{if }v\in\Sigma_{K,\infty}\\
       H^1_{cris}(\cE_{i,s,k_v}/W(k_v))\otimes_{W(k_v)} F_v &\quad\text{if }v\in\Sigma_{K,f},
     \end{cases}
\end{align*}
where the former denotes singular cohomology with $\bQ$-coefficients, and the latter denotes crystalline cohomology with $W(k_v)$-coefficients \cite{Berthelot1974} tensored with $F_v$. Let $H^1_{DR}(\cE_{i,s}/K_v)$ denote the de Rham cohomology of $\cE_{i,s,K_v}$. Note that $H^1_{DR}(\cE_{i,s}/K_v)\cong H^1_{DR}(\cE_{i,s}^\van)$ (\cite[Cor. 31.1.2, 32.2.2]{ABC}). Define the analogous objects for $s$ replaced with $s_0$.
For $i\in\{1,2\}$, let $\pi_i:\cE_i\to C$ denote the structure map.

\begin{lemma}\label{lem:commutes}
For $t\in\{s,s_0\}$, there exist canonical (comparison) isomorphisms
    \[\sigma_i:H^1_{DR}(\cE_{i,t}/K_v)\to H^1_v(\cE_{i,t})\otimes_{F_v} K_v\]
belonging to commutative diagrams
\begin{center}
\begin{tikzcd}
{H^1_{DR}(\cE_{i,s_0}/K_v)} \arrow[d, "\sigma_i" ] \arrow[r, "\epsilon_i"] & {H^1_{DR}(\cE_{i,s}/K_v)}\arrow[d, "\sigma_i"] \\
 {H^1_v(\cE_{i,s_0})\otimes_{F_v} K_v} \arrow[r, "\iota_i"] & {H^1_v(\cE_{i,s})\otimes_{F_v} K_v},
\end{tikzcd}
\end{center}
where $\epsilon_i$ is the isomorphism induced by the Gauss--Manin connection (``parallel transport'') and $\iota_i$ is induced

(if $v\in\Sigma_{K,\infty}$) by the canonical isomorphism $R^1(\pi^\van_{i})_*\bQ|_{s_0}\cong R^1(\pi^\van_{i})_*\bQ|_{s}$ of fibres of the constant sheaf $R^1(\pi^\van_{i})_*\bQ|_{\Delta_v}$;

(if $v\in\Sigma_{K,f}$) by the pullback $\alpha^*_i$ 
       associated with an isogeny $\alpha_i:\cE_{i,s,k_v}\to\cE_{i,s_0,k_v}$ whose degree we denote $d_i$.  

\end{lemma}

\begin{proof}
    If $v\in\Sigma_{K,\infty}$, the maps $\cE^\van_i|_{\Delta_v}\to\Delta_v$ are smooth and proper, and so the claims follow from \cite[Prop.~4.1.2]{Katz-p-curvature}.
    
    If $v\in\Sigma_{K,f}$, the $\sigma_i$ are afforded to us by \cite[Cor.~2.5]{BO1983}. Furthermore, in the terminology of \cite[\S5]{Ogus84}, $\cE_{i,K_v}$ has {\it good reduction over $\fC_{\cO_{K_v}}$}. Since $\Delta_v$ is simple, the $\cO_K$-extensions of $s$ and~$s_0$ have the same reduction at~$v$.   
    Therefore, the commutative diagrams in this case are afforded to us by \cite[Rem.~5.14.3]{Ogus84}.
\end{proof}

Let $\{\gamma_{v,2,s_0},\delta_{v,2,s_0}\}$ denote
a $K_v$-basis for $H^1_{DR}(\cE_{2,s_0}/K_v)$
whose elements are contained in the $F_v$-subspace $\sigma^{-1}_2(H^1_v(\cE_{2,s_0}))$,
and let
\begin{equation} \label{eqn:gamma-f0}
\gamma_{v,1,s_0}:=f^{*}_0\gamma_{v,2,s_0},\ \delta_{v,1,s_0}:=f^{*}_0\delta_{v,2,s_0}.
\end{equation}
Since $f_0$ is a $K$-isogeny, $\gamma_{v,1,s_0}$ and $\delta_{v,2,s_0}$ form a $K_v$-basis for $H^1_{DR}(\cE_{1,s_0}/K_v)$.

We claim there is a commutative diagram
\begin{center}
\begin{tikzcd}
{H^1_{DR}(\cE_{2,s_0}/K_v)} \arrow[d, "\sigma_2" ] \arrow[r, "f^*_0"] & {H^1_{DR}(\cE_{1,s_0}/K_v)}\arrow[d, "\sigma_1"] \\
 {H^1_v(\cE_{2,s_0})\otimes_{F_v} K_v} \arrow[r, "f^*_0"] & {H^1_v(\cE_{1,s_0})\otimes_{F_v} K_v}.
\end{tikzcd}
\end{center}
From this, we conclude $\gamma_{v,1,s_0},\delta_{v,1,s_0}\in\sigma^{-1}_1(H^1_v(\cE_{1,s_0}))$. For $v\in\Sigma_{K,\infty}$, the existence of the diagram is standard. For $v\in\Sigma_{K,f}$, we consider $\mathfrak{t}\in\{\fs_0,\fs\}$ and write $t:=\mathfrak{t}_K$. Then, the semiabelian scheme $\fG_{i,\mathfrak{t}}$ is the connected N\'eron model of its generic fibre $\cE_{i,t}$ and so, by the N\'eron mapping property, we obtain an isogeny
\[f_{0,k_v}:\cE_{1,s_0,k_v}\to \cE_{2,s_0,k_v}\]
(and similarly $f_{k_v}:\cE_{1,s,k_v}\to\cE_{2,s,k_v}$). This induces the morphism along the bottom of the diagram. The fact that the diagram commutes in this case is \cite[(2.4.5)]{BO1983}.

The parallel transport isomorphisms $\epsilon_i$ allow us to extend the $\{\gamma_{v,i,s_0},\delta_{v,i,s_0}\}$ to horizontal bases $\{\gamma_{v,i},\delta_{v,i}\}$ for $H^1_{DR}(\cE_i/C_0)^\van(\Delta_v)$ and we let $\Omega_{v,i}\in\rM_2(\cO_{\Delta_v})$ be the matrix whose rows give the coordinates of $\{\omega^\van_i,\eta^\van_i\}$ in terms of $\{\gamma_{v,i},\delta_{v,i}\}$. We define $\theta_{v,i}:=\Omega_{v,i}(s^\van_0)$ and $Y_{v,i}:=\Omega_{v,i}\cdot\theta_{v,i}^{-1}$. By construction, the Taylor series of the entries of the $Y_{v,i}$ with respect to $x$, are precisely the $F_{imn}$ invoked in Section \ref{sec:pairs}. By \eqref{eqn:omega-f0} and~\eqref{eqn:gamma-f0}, we have
\begin{equation} \label{eqn:theta-12}
\theta_{v,1}=\mu\theta_{v,2}.
\end{equation} 

\begin{remark}\label{rem:G-fn}
The argument above establishes that the $F_{imn}$ have $v$-adic radius of convergence at least~$1$ at almost all $v\in\Sigma_K$.
By Lemma \ref{lem:horizontal-sections}, they have positive $v$-adic radius of convergence at the finitely many remaining $v\in\Sigma_K$.
This may look similar to the proof of \cite[Thm.~B]{And89} in \cite[V, App.]{And89}, but this is not sufficient to establish that the $F_{imn}$ are G-functions.
The proof in \cite[V, App.]{And89} uses the stronger fact that the $v$-adic \emph{generic} radius of convergence of the differential equation satisfied by the $F_{imn}$ is~$1$ for almost all~$v\in\Sigma_K$.
\end{remark}

\paragraph{\bf Relations}

The set-up above did not use the existence of an isogeny $f \colon \cE_{1,s} \to \cE_{2,s}$.
Now we use this isogeny to construct relations between evaluations of elements of~$\cG$.

To that end, write 
\[f^*\omega_{2,s}=a\omega_{1,s}+c\eta_{1,s}\text{ and } f^*\eta_{2,s}=b\omega_{1,s}+d\eta_{1,s}\] for some $a,b,c,d\in K$, and
\[f^{\van *}\gamma_{v,2,s}=p_v\gamma_{v,1,s}+q_v\delta_{v,1,s}\text{ and } f^{\van *}\delta_{v,2,s}=r_v\gamma_{v,1,s}+s_v\delta_{v,1,s}\] 
for some $p_v,q_v,r_v,s_v\in F_v$.
We consider these coefficients as forming matrices
\[A:=\fullmatrix{a}{b}{c}{d}\text{ and }M_v:=\fullmatrix{p_v}{r_v}{q_v}{s_v}.\]

\smallskip

\begin{lemma}\label{lem:trace}
  We have $\tr(M_v)\in\mathbb{Q}$. 
\end{lemma}

\begin{proof}
Again, we have a commutative diagram
\begin{center}
\begin{tikzcd}
{H^1_{DR}(\cE_{2,s}/K_v)} \arrow[d, "\sigma_2" ] \arrow[r, "f^*"] & {H^1_{DR}(\cE_{1,s}/K_v)}\arrow[d, "\sigma_1"] \\
 {H^1_v(\cE_{2,s})\otimes_{F_v} K_v} \arrow[r, "f^*"] & {H^1_v(\cE_{1,s})\otimes_{F_v} K_v}.
\end{tikzcd}
\end{center}
If $v\in\Sigma_{K,\infty}$, the bottom arrow is the extension to $K_v$ of an isomorphism $H^1_v(\cE_{2,s})\to H^1_v(\cE_{1,s})$ of $\bQ$-vector spaces. By Lemma \ref{lem:commutes}, the sets \[\{\sigma_2(\gamma_{v,2,s}),\sigma_2(\delta_{v,2,s)}\}\text{ and }\{\sigma_1(\gamma_{v,1,s}),\sigma_1(\delta_{v,2,s})\}\] are bases $B_2$ and $B_1$ for $H^1_v(\cE_{2,s})$ and $H^1_v(\cE_{1,s})$, respectively. The matrix representing the lower $f^*$ with respect to $B_2$ and $B_1$ is $M_v$, and so $M_v\in\rM_2(\bQ)$. In particular, $\tr(M_v)\in\bQ$, as claimed.

For $v\in\Sigma_{K,f}$, we write $f^\vee_{0,k_v}:\cE_{2,s_0,k_v}\to\cE_{1,s_0,k_v}$ for the dual isogeny to $f_{0,k_v}$. Then, using the isogenies $\alpha_1$, $\alpha_2$ from \cref{lem:commutes}, the composition
\begin{center}
    \begin{tikzcd}
{\cE_{2,s_0,k_v}} \arrow[r, "f^\vee_{0,k_v}"] & {\cE_{1,s_0,k_v}} \arrow[r, "\alpha_1^\vee"] & {\cE_{1,s,k_v}} \arrow[r, "f_{k_v}"] & {\cE_{2,s,k_v}} \arrow[r, "\alpha_2"] & {\cE_{2,s_0,k_v}}
\end{tikzcd}
\end{center}
yields an endomorphism of $\cE_{2,s_0,k_v}$ and hence an endomorphism of $H^1_v(\cE_{2,s_0})$, which we denote $\phi$. By Lemma \ref{lem:commutes} and the naturality of the~$\sigma_i$, we obtain the commutative diagram
\begin{center}\begin{tikzcd}
    H^1_{DR}(\cE_{2,s_0}) \ar[d, "\sigma_2"] \ar[r, "\epsilon_2"]
  & H^1_{DR}(\cE_{2,s})   \ar[d, "\sigma_2"] \ar[r, "f^*"]
  & H^1_{DR}(\cE_{1,s})   \ar[d, "\sigma_1"] \ar[r, "d_1\epsilon_1^{-1}"]
  & H^1_{DR}(\cE_{1,s_0}) \ar[d, "\sigma_1"] \ar[r, "N_0(f_0^*)^{-1}"]
  & H^1_{DR}(\cE_{2,s_0}) \ar[d, "\sigma_2"]
\\  H^1_v(\cE_{2,s_0})  \ar[r, "\iota_2"] \ar[rrrr, bend right=15, "\phi"]
  & H^1_v(\cE_{2,s})    \ar[r, "f^*"]
  & H^1_v(\cE_{1,s})    \ar[r, "d_1\iota_1^{-1}"]
  & H^1_v(\cE_{1,s_0})  \ar[r, "N_0(f_{0}^*)^{-1}"]
  & H^1_v(\cE_{2,s_0}).
\end{tikzcd} \end{center}
By following $\gamma_{v,2,s_0}$ and $\delta_{v,2,s_0}$ along the top line of this diagram, we conclude that,
with respect to the basis $\{\sigma_2(\gamma_{v,2,s_0}),\sigma_2(\delta_{v,2,s_0})\}$, the matrix representing $\phi$ is $d_1N_0M_v$.

Therefore, since $H^1_{cris}(\cE_{2,s_0,k_v}/W(k_v))$ is canonically isomorphic to the Dieudonn\'e module associated with the $p$-divisible group $\cE_{2,s_0,k_v}[p^\infty]$ \cite[Rem.~7.3.3]{BC:CMI}, we conclude from \cite[V.2, Corollary]{demazure} that $\tr(d_1N_0M_v)\in\bZ$. We obtain $\tr(M_v)\in\bQ$, as claimed.
\end{proof}




Resuming from the passage preceding Lemma \ref{lem:trace}, we have
\[A^T\Omega_{v,1}(s^\van)=\Omega_{v,2}(s^\van)M^T_v.\] 
Since the Taylor series of an analytic function $\varphi$ converges to $\varphi$ inside the radius of convergence, if we write $F_{i,s}:=(F^\van_{i}(x^\van(s^\van))_{mn})_{1\leq m,n\leq 2}$, then we obtain
\[ A^T F_{1,s} \theta_{v,1} = F_{2,s} \theta_{v,2} M_v^T. \]
Using~\eqref{eqn:theta-12}, this rearranges to
\begin{equation} \label{eqn:Mv-conjugation}
\mu(F_{2,s})^{-1} A^T F_{1,s} = \theta_{v,2} M^T_v \theta_{v,2}^{-1}.
\end{equation}
Using the invariance of trace under conjugation, this yields
\[\tr(\mu(F_{2,s})^{-1}A^TF_{1,s}) = \tr(M^T_v) = \tr(M_v),\]
and, finally, multiplying both sides by $\det(F_{2,s})$, we arrive at
\[ \tr(\mu F_{2,s}^{\rm adj}A^TF_{1,s}) = \det(F_{2,s})\tr(M_v).\]
This gives a $v$-adic relation $Q_v$ of degree $2$ between the evaluations of the elements of $\cG$ at $x(s)$. The fact that $Q_v\notin I$ is readily seen by observing that $Q_v$ has degree~$1$ with respect to the entries of~$F_{1,s}$, while $I$ is generated by a polynomial of degree~$2$ with respect to $X_{1mn}$ ($1 \leq m,n \leq 2$).

\subsubsection{Case 2}\label{sec:Case2}

Now suppose $v\in\Sigma_{K,f}$ and that the $\cE_{i,s_0}$ (equivalently, the $\cE_{i,s}$) have bad reduction at $v$. Thus, after possibly replacing $K$ with a finite extension, we have $\fG_{i,\fs_0,k_v}\cong\fG_{i,\fs,k_v}\cong\bG_{m,k_v}$.
At such a place, we may interpret the $v$-adic evaluations of the G-functions associated with $\cE_i \to C_0$ in a similar way to the ``locally invariant periods'' around a point of multiplicative degeneration, as in \cite{Y(1)} or \cite{LGO}.

More precisely, we can put ourselves in the situation of \cite[3.B.1]{LGO} using the following dictionary, which sends the notations of \cite[3.B.1]{LGO} to those of this section:
\begin{align} \label{eqn:dictionary}
        R\mapsto  \ \cO_{K_v},\ 
        \fC\mapsto  \ \fC_{\cO_{K_v}},\ 
        \fC_{0,\fp}\mapsto  \ \fs_0(\fp_v),\ 
        \fG\mapsto & \ \fG_{\cO_{K_v}.}
\end{align}

Let $\cC_v$ denote the residue disc $\red_v^{-1}(\fs_0(\fp_v)) \subset \hat{\fC}_{\fC_{k_v}}^\rig \subset C^\van$ (where, as before, $\hat{\fC}_{\fC_{k_v}}$ denotes the formal completion of $\fC$ along $\fC_{k_v}$). By \cite[Prop.~0.2.7]{Berthelot96}, $\cC_v$ is equal to the image $\hat{\fC}_{\fs_0(\fp_v)}^\rig\subset C^\van$ of the formal completion $\hat{\fC}_{\fs_0(\fp_v)}$ of $\fC$ along $\fs_0(\fp_v)$ under Berthelot's rigid generic fibre functor.
Thus, via the dictionary~\eqref{eqn:dictionary}, the rigid space denoted $\cC = \fC_{\formal\dag}^\rig$ in \cite[3.B.1]{LGO} corresponds to~$\cC_v$.

Let $\cE_{v,i}:=\cE^\van_i|_{\cC_v}$. By \cite[Prop.~3.1 and Lemma~3.3]{LGO}, there exist rigid analytic uniformisations $\phi_{v,i}:\bG^\van_m\times\cC_v\to\cE_{v,i}$.
(Note that we do not require any compatibility between the rigid uniformisations at different places~$v$, essentially because the functions $Y_v$ defined in Lemma~\ref{lem:GF-ind-v} are independent of the choice of horizontal sections~$\gamma_{v,i}$.)
Moreover, by \cite[Satz 5]{Ger70}, we obtain commutative diagrams
\[ 
\xymatrix{
    \bG_m^\van    \ar[d]^{\phi_{v,1,s_0}} \ar[r]^{[m_0]}
  & \bG_m^\van   \ar[d]^{\phi_{v,2,s_0}}
\\  \cE_{v,1,s_0}                              \ar[r]^{f^\van_0}
  & \cE_{v,2,s_0}
} 
\hspace{1cm}\xymatrix{
    \bG_m^\van    \ar[d]^{\phi_{v,1,s}} \ar[r]^{[m]}
  & \bG_m^\van   \ar[d]^{\phi_{v,2,s}}
\\  \cE_{v,1,s}                              \ar[r]^{f^\van}
  & \cE_{v,2,s}
}
\]
for natural numbers $m_0|N_0$ and $m|\deg(f)$, from which we obtain commutative diagrams
\begin{equation} \label{eqn:unif-pullback-diagrams}
\xymatrix{
    H^1_{DR}(\cE_{v,2,s_0})    \ar[d]^{\phi^*_{v,2,s_0}} \ar[r]^{(f^\van_0)^*}
  & H^1_{DR}(\cE_{v,1,s_0})   \ar[d]^{\phi^*_{v,1,s_0}}
\\  H^1_{DR}(\bG^\van_m)                              \ar[r]^{[m_0]}
  & H^1_{DR}(\bG^\van_m)
}
\hspace{1cm}
\xymatrix{
    H^1_{DR}(\cE_{v,2,s})    \ar[d]^{\phi^*_{v,2,s}} \ar[r]^{(f^\van)^*}
  & H^1_{DR}(\cE_{v,1,s})   \ar[d]^{\phi^*_{v,1,s}}
\\  H^1_{DR}(\bG^\van_m)                              \ar[r]^{[m]}
  & H^1_{DR}(\bG^\van_m).
}
\end{equation}

Since $\cE_i\to C_0$ is proper, the natural map
\[H^1_{DR}(\cE_i/C_0)^\van\to H^1_{DR}(\cE^\van_i/C^\van_0)\]
of modules with integrable connection is an isomorphism by standard applications of rigid analytic GAGA (or, more precisely, \cite[Appendix A, (A.1.1)]{Conrad:relative}). 
We also have the pullback
\[\phi_{v,i}^*|_{\cC_v}:H^1_{DR}(\cE_{v,i}/\cC_v)\to H^1_{DR}(\bG^\van_m\times\cC_v/\cC_v)\]
(which is non-trivial because $\phi_{v,i}$ is smooth), for which the target is isomorphic to \[H^1_{DR}(\bG^\van_m)\otimes_{K_v}\cO_{\cC_v}= \cO_{\cC_v}\cdot dz/z.\]
(Indeed, $\bG_m \times C_0 \to C_0$ is a rational elementary fibration \cite[Def.~25.1.4]{ABC}, hence, its Artin set is equal to~$C_0$ \cite[Rmk.~25.2.3]{ABC}. Therefore, \cite[Thm.~32.2.1]{ABC} gives us this isomorphism.) 

Since $\Delta_v$ is simple, we have $\Delta_v \subset \cC_v$.
By \cref{lem:horizontal-sections}, after shrinking $\Delta_v$ and~$R_v$, we may assume that $H^1_{DR}(\cE_2/C)^\van(\Delta_v)$ possesses a basis of horizontal sections.
It follows that the restriction of $\phi_{v,2}^*|_{\Delta_v}$ to horizontal sections is a non-zero linear map from a $K_v$-vector space of dimension~$2$ to a $K_v$-vector space of dimension~$1$.
Therefore, we can choose a horizontal basis $\{\gamma_{v,2},\delta_{v,2}\}$ for $H^1_{DR}(\cE_2/C)^\van(\Delta_v)$ such that
\[\phi_{v,2}^*(\gamma_{v,2}) = dz/z, \quad \phi_{v,2}^*(\delta_{v,2}) = 0.\]

As in Section \ref{sec:Case1}, define the basis
\[\gamma_{v,1,s_0}:=(f^{\van}_0)^*\gamma_{v,2,s_0},\ \delta_{v,1,s_0}:=(f^{\van}_0)^*\delta_{v,2,s_0}\] 
for $H^1_{DR}(\cE_{v,1,s_0})$.
Again by \cref{lem:commutes}, after shrinking $\Delta_v$ and~$R_v$, we may assume that $H^1_{DR}(\cE_1/C)^\van(\Delta_v)$ possesses a basis of horizontal sections, so that $\{ \gamma_{v,1,s_0}, \delta_{v,1,s_0} \}$
extends to a horizontal basis $\{\gamma_{v,1},\delta_{v,1}\}$ for $H^1_{DR}(\cE_1/C_0)^\van(\Delta_v)$.
By \eqref{eqn:unif-pullback-diagrams} and the fact that horizontal sections form a local system, we have
\[ \phi_{v,1}^*(\gamma_{v,1}) = m_0\, dz/z, \quad \phi_{v,1}^*(\delta_{v,1}) = 0. \]

Using the right-hand diagram from~\eqref{eqn:unif-pullback-diagrams}, with $A$ and $\Omega_{v,i}$ defined as in Section \ref{sec:Case1}, we arrive at the following equation:
\[m_0\,A^T\Omega_{v,1}(s^\van)(dz/z,0)^T=m\, \Omega_{v,2}(s^\van)(dz/z, 0)^T\]
(where we write $(dz/z, 0)^T$ for the corresponding column vector).
If $(\theta_{v,2})_{21}\neq 0$, we write $\lambda_{v,2}:=(\theta_{v,2})_{11}(\theta_{v,2})^{-1}_{21}$ and, expanding, we obtain
\begin{align}
\label{eqn:ac-rel}
    a\mu(F^\van_{1,s})_{12}+c\mu(F^\van_{1,s})_{22}-\frac{m}{m_0}(F^\van_{2,s})_{12}&=\lambda_{v,2}[\frac{m}{m_0}(F^\van_{2,s})_{11}-a\mu(F^\van_{1,s})_{11}-c\mu(F^\van_{1,s})_{21}]\\
\label{eqn:bd-rel}
    b\mu(F^\van_{1,s})_{12}+d\mu(F^\van_{1,s})_{22}-\frac{m}{m_0}(F^\van_{2,s})_{22}&=\lambda_{v,2}[\frac{m}{m_0}(F^\van_{2,s})_{21}-b\mu(F^\van_{1,s})_{11}-d\mu(F^\van_{1,s})_{21}],
\end{align}
where, again, $(F^\van_{i,s})_{jk}:=(F^\van_{i}(x^\van(s^\van)))_{jk}$.
If the right-hand side of~\eqref{eqn:ac-rel} is $0$, we obtain a linear $v$-adic relation between the evaluations of the elements of $\cG$ at $x(s)$. Otherwise, we can divide \eqref{eqn:bd-rel} by~\eqref{eqn:ac-rel} and clear denominators to obtain a $v$-adic relation of degree~$2$ between the evaluations of the elements of $\cG$ at $x(s)$. 
If $(\theta_{v,2})_{21}= 0$, we must have $(\theta_{v,2})_{11}\neq 0$, and we obtain
\begin{align}
    \frac{m}{m_0}(F^\van_{2,s})_{11}-a\mu(F^\van_{1,s})_{11}-c\mu(F^\van_{1,s})_{21}=0
\end{align}

Again, the fact that these relations are not in $I$ is readily seen by elementary computation. 
\qed

\subsection{Isogenous triples}\label{sec:triples}

In this section, we introduce a third elliptic scheme in order to obtain a single relation holding at all places of good ordinary reduction. That is, we repeat the setup of Section \ref{sec:pairs} but with the following modifications: we let $i\in\{1,2,3\}$, we suppose that there are isogenies 
\[f_0:\cE_{1,s_0}\to\cE_{2,s_0}\text{ and }f'_0:\cE_{1,s_0}\to\cE_{3,s_0}\]
of degrees $N_0:=\deg(f_0)$ and $N'_0:=\deg(f'_0)$, and we impose the analogous conditions on the bases for each $H^1_{DR}(\cE_i/C_0)$, with constants $\mu,\mu'\in\bQ$, corresponding to $f_0$ and $f'_0$, respectively. We will use the notation of Section \ref{sec:Case1}.

The result is as follows.

\begin{proposition}\label{prop:rel2}
    Let $s\in C(L)$, with $L$ a finite extension of $K$, such that there exist isogenies \[f:\cE_{1,s}\to\cE_{2,s}\text{ and }f':\cE_{1,s}\to\cE_{3,s}.\]
    
    Let $v\in\Sigma_{K}$ and choose $w\in\Sigma_{L}$ lying over $v$. 
    
    If $s^\wan\in \Delta_v$, then there exists a $w$-adic relation $Q_{w}$ of degree at most $4$ between the evaluations of the elements of $\cG$ at $x(s)$ which is not contained in the ideal $I$ of $\Qbar[X_{i mn}:1\leq i\leq 3,\  1\leq m,n\leq 2]$ generated by the elements 
    \[\det((X_{1 mn})_{mn})-\det((X_{2 mn})_{mn})\text{ and }\det((X_{1 mn})_{mn})-\det((X_{3 mn})_{mn}).\]
    Moreover, for $w\in\Sigma_{L,f}$ for which the $\cE_{i,s,k_v}$ are ordinary, we can choose $Q_{w}$ independently of $v$. Otherwise, we can choose $Q_{w}$ of degree at most $2$.
\end{proposition}

\begin{proof}
The existence of $Q_w$ with $\deg(Q_w) \leq 2$ for $w \not\in \Sigma_{L,f}$ follows immediately from Proposition~\ref{prop:rel} (we may ignore $\cE_3$ for these parts of the proposition).

It remains to prove that, for those $w \in \Sigma_{L,f}$ for which $\cE_{i,s,k_v}$ are ordinary, we can choose $Q_w$ independently of $w$ and of degree at most~$4$.
Consider such a place~$w$.
Replace $K$ with $L$ and $v$ with $w$. 
    
In this case, $E:=\End(\cE_{2,s,k_v})\otimes_\bZ\bQ$ is an imaginary quadratic field. After possibly replacing $K$ with a finite extension (which may depend on $v$ but will not be visible in our sought relation), we may choose $\{\gamma_{v,2,s_0},\delta_{v,2,s_0}\}$ so that the action of $E$ on $H^1_v(\cE_{2,s})$ is diagonal. This $\{\gamma_{v,2,s_0},\delta_{v,2,s_0}\}$ determines $\{\gamma_{v,3,s_0},\delta_{v,3,s_0}\}$ via the conditions
\[ \gamma_{v,1,s_0} = f_0^* \gamma_{v,2,s_0} = f_0'^* \gamma_{v,3,s_0},
\quad \delta_{v,1,s_0} = f_0^* \delta_{v,2,s_0} = f_0'^* \delta_{v,3,s_0}. \]
The composed morphism
\[ H^1_v(\cE_{2,s}) \xrightarrow{\iota_2^{-1}} H^1_v(\cE_{2,s_0}) \xrightarrow{f_0^*} H^1_v(\cE_{1,s_0}) \xrightarrow{(f_0'^*)^{-1}} H^1_v(\cE_{3,s_0}) \xrightarrow{\iota_3} H^1_v(\cE_{3,s}) \]
conjugates the action of $E = \End(\cE_{2,s,k_v}) \otimes_\bZ \bQ$ on~$H^1_v(\cE_{2,s,k_v})$ into the action of $\End(\cE_{3,s,k_v}) \otimes_\bZ \bQ \cong E$ on~$H^1_v(\cE_{3,s})$.
Hence the latter action is also diagonal, with respect to the basis $\{\gamma_{v,3,s_0},\delta_{v,3,s_0}\}$.

If we denote by $A$, $A'$, $M_v$ and $M'_v$ the matrices representing $f$ and $f'$ (as above), we conclude, as in the proof of Lemma \ref{lem:trace}, that the corresponding endormorphisms $\phi$ and $\phi'$ on $H^1_v(\cE_{2,s})$ and $H^1_v(\cE_{3,s})$ are represented by $d_1N_0M_v$ and $d_1N'_0M'_v$, respectively ($d_1$ as defined in Lemma \ref{lem:commutes}). Hence, we conclude that $M_v$ and $M'_v$ are diagonal.

Again, as at~\eqref{eqn:Mv-conjugation}, we obtain
    \begin{align*}
        \Lambda_v&:=\mu(F_{2,s})^{\rm adj}A^TF_{1,s}=\det(F_{2,s})\theta_{v,2} M_v\theta_{v,2}^{-1};\\
        \Lambda'_v&:=\mu'(F_{3,s})^{\rm adj}(A')^TF_{1,s}=\det(F_{3,s})\theta_{v,3} M'_v\theta_{v,3}^{-1}.
    \end{align*}
Since $M_v$ and $M'_v$ are diagonal and $\theta_{v,3}=\mu(\mu')^{-1}\theta_{v,2}$, the matrices $\Lambda_v$ and $\Lambda'_v$ commute. In particular,
\begin{equation} \label{eqn:Lambda-commute-relation}
\Lambda_{v,21}\Lambda'_{v,12}=\Lambda_{v,12}\Lambda'_{v,21}.
\end{equation}
Each entry of~$\Lambda_v$ or~$\Lambda'_v$ is the evaluation at $(F_{imn}^\van(x(s)^\van)_{1 \leq i \leq 3,\, 1 \leq m,n \leq 2}$ of a polynomial of degree $4$, with coefficients in~$K$, which is independent of~$v$.
Thus, \eqref{eqn:Lambda-commute-relation} gives us the desired $v$-adic relation between evaluations of~$\cG$ at~$s$, independent of~$v$.

Again, the fact that these relations are not in $I$ can be seen by elementary computation. For example, since $A$ and $A'$ are invertible, we can find $\gamma_2, \gamma_3 \in \SL_2(K)$ such that
\[ \gamma_2^{\rm adj} A^T = \fullmatrix{\det(A)}{1}{0}{1}\text{ and }\, \gamma_3^{\rm adj} (A')^T = \fullmatrix{\det(A')}{0}{1}{1}. \]
Define $\gamma_1$ to be the identity matrix so, in particular, we have  
\[\det(\gamma_1) = \det(\gamma_2) = \det(\gamma_3).\]
Hence, $(\gamma_1,\gamma_2,\gamma_3)\in V(I)$. However, we have 
\begin{align*}
   \Lambda_v(\gamma_1,\gamma_2,\gamma_3) & = \mu \gamma_2^{\rm adj} A^T \gamma_1 = \mu \fullmatrix{\det(A)}{1}{0}{1},
\\ \Lambda'_v(\gamma_1,\gamma_2,\gamma_3) & =  \mu' \gamma_3^{\rm adj} (A')^T \gamma_1 = \mu' \fullmatrix{\det(A')}{0}{1}{1}.
\end{align*}
Thus, $(\gamma_1,\gamma_2,\gamma_3)$ does not satisfy~\eqref{eqn:Lambda-commute-relation}.
\end{proof}

\section{Proving Theorems \ref{teo:main-ht} and \ref{teo:super-ht}}\label{sec:proofs}

In this section, we prove the main height bounds, namely, Theorems \ref{teo:main-ht} and \ref{teo:super-ht}. Before doing so, we collect some general constructions that we will use in the proofs.

\subsection{Elliptic curve schemes}\label{sec:ecs} 
Let $\cE\to S:=\bA^1_\bQ\setminus\{0,1728\}$ denote the ``$j$-family'' of elliptic curves defined by the equation 
\[y^2+xy=x^3-\frac{36}{j-1728}x-\frac{1}{j-1728}.\]
As per the name, for $j_0\in S(\Qbar)$, the fiber $\cE_{j_0}$ is an elliptic curve with $j$-invariant $j(\cE_{j_0})=j_0$. The fiber at $\infty$ of the Zariski closure of~$\cE$ in $\bP^1 \times \bP^2$ is a nodal cubic, so the fibre at~$\infty$ of the connected N\'eron model of~$\cE$ over $\bP^1_\bQ\setminus\{0,1728\}$ is isomorphic to $\bG_m$.

Now let $n\in\mathbb{N}$ and suppose that $C_0\subset S_{K}^n$ is a geometrically irreducible algebraic curve over a number field $K$. For $i\in\{1,\ldots, n\}$, we obtain elliptic schemes $\cE_{i,0}\to C_0$ by pulling back $\cE_K\to S_K$ along the co-ordinate projections $C_0\to S_{K}$. We refer to the $\cE_{i,0}\to C_0$ as the {\bf standard elliptic curve schemes on $C_0$}.

\subsection{Covering data}\label{sec:cover}
Consider the situation in Section \ref{sec:ecs} and let $C$ denote a smooth projective model of $C_0$.
After possibly replacing $K$ with a finite extension and $C$ with a finite \'etale cover, we may assume that the connected N\'eron model $\cE_i$ of $\cE_{i,0}$ over $C$ is semiabelian \cite[Sec.~7.4, Thm.~1]{BLR90}.

Let $s_0\in C(K)$. By \cite[Lemma 6.5]{LGO}, after possibly replacing $K$ with a finite extension, we obtain 
\begin{enumerate}[(a)]
\item a smooth projective geometrically irreducible algebraic curve $\tilde{C}$ over $K$;
\item a non-constant morphism $\nu:\tilde{C}\to C$ over $K$;
\item a projective regular $\cO_K$-model ${\fC}$ of $\tilde{C}$;
\item semiabelian schemes $\fG_i\to{\fC}$ for $i\in\{1,\dotsc,n\}$;
\item a non-constant rational function $x\in K(\tilde{C})$,
\end{enumerate}
such that
\begin{enumerate}[(i)]
\item for each~$i\in\{1,\ldots,n\}$, $\fG_{i,K}\cong\tilde{\cE}_i:=\tilde{C}\times_{C}\cE_i$;
\item every point $s\in \tilde{C}(\overline{K})$ satisfying $x(s) = 0$ is a simple zero of $x$ and satisfies $\nu(s) = s_0$;
\item $x:\tilde{C}\to\bP^1$ is Galois, which is to say, $\Aut_x(\tilde{C})= \{\sigma\in\Aut(\tilde{C}) : x \circ \sigma = x\}$ acts
transitively on each $\overline{K}$-fibre of $x$.
\end{enumerate}
We may replace $K$ with a further finite extension so that the zeros $s_1,\ldots,s_\ell\in\tilde{C}(\overline{K})$ of $x$ belong to $\tilde{C}(K)$. Note that, by (ii), $x$ is a local parameter at $s_k$ for all $k\in\{1,\ldots,\ell\}$. By (iii), for each $k$, there is an element $\sigma_k\in\Aut_x(\tilde{C})$ such that $\sigma_k(s_1)=s_k$.

By \cite[Lemma 6.2]{LGO}, there exists another regular projective $\cO_K$-model ${\fC}'$ of $\tilde{C}$ such that the $\sigma_k$ extend to morphisms ${\fC}'\to{\fC}$. Let $\tilde{C}_0=\nu^{-1}(C_0)$. We refer to \[\cD:=(\fG_1,\ldots,\fG_n,{\fC}',{\fC},\tilde{C},\tilde{C}_0,\nu,x,s_1,\dots,s_\ell,\sigma_1,\dots,\sigma_\ell)\] as a {\bf covering datum for $(C,s_0)$}.

\subsection{Pullback representatives}\label{sec:reps}

Continue from the situation obtained in Section~\ref{sec:cover}.
Fix a set of indices $I \subset \{1,\dotsc,n\}$ (this is for convenience, enabling us to have a common set-up for the proof of \cref{teo:main-ht}, which involves two indices, and the proof of \cref{teo:super-ht}, which involves three indices).
For $\lambda=(k,i)\in \{ 1, \dotsc, \ell \}\times I$, we write $\fG_\lambda:=\sigma^*_k\fG_i$ and we write $\cE_\lambda:=\fG_{\lambda,K}$. Then $\fG_\lambda$ is a scheme over ${\fC}'$ and $\cE_\lambda$ is a scheme over ${\fC}'_K=\tilde{C}$.

Let $\bar\eta$ denote a geometric generic point of $\tilde{C}$ and define an equivalence relation $\sim$ on $\{ 1, \dotsc, \ell \}\times I$ by the condition
\[ \lambda \sim \mu \text{ if there exists an isogeny } \cE_{\lambda,\bar\eta} \to \cE_{\mu,\bar\eta}.\]
Let $\Lambda$ denote a set of representatives for the induced equivalence classes. We refer to $\Lambda$ as a {\bf set of pullback representatives for $(\cD,I)$}. We will write $[k,i]$ for the unique element of $\Lambda$ equivalent to $(k,i)$.

\subsection{Neighbourhood systems} \label{sec:neighourhood-systems}

Continue from the situation obtained in Section \ref{sec:reps}. To simplify notation, relabel $C:=\tilde{C}$. For each archimedean (resp.\ non-archimedean) place $v$ of $K$, apply \cite[Lemma 5.5]{Y(1)} (resp.\ this paper's Proposition \ref{rigid-discs}) to $C$, $\fC$ and $x\in K(C)$, thereby obtaining real numbers $r_v > 0$ and simply connected (resp. simple) open subspaces $U_{v,1}, \dotsc, U_{v,\ell}$ of $C^\van$ with the properties (i)--(v) of \cite[Lemma 5.5]{Y(1)} (resp.\ Proposition \ref{rigid-discs}). We refer to $\cS:=((r_v,U_{v,1}, \dotsc, U_{v,\ell})_v)$ as a {\bf neighbourhood system for $(C,x)$}.

Let $\zeta\in K^\times$ such that $|\zeta|_v\leq r_v$ for the finitely many places $v$ of $K$ for which $r_v<1$ and let $H$ denote the G-function $\zeta/(\zeta-X)$. We refer to $H$ as an {\bf auxillary G-function} for $\cS$.

\subsection{Proof of Theorem \ref{teo:main-ht}}\label{proof:main-ht}

Consider the situation in Theorem \ref{teo:main-ht} and let $K$ denote the field of definition of $C \subset Y(1)^3$. To match the notation of the preceding sections, we relabel $C$ and $\overline{C}$ from the statement of Theorem~\ref{teo:main-ht} as $C_0$ and~$C$, respectively.
We break the proof into sections.

\subsubsection{Reductions}
After possibly relabeling the coordinates, we can assume that $C\setminus C_0$ contains a point $s_0:=(b_1,b_2,\infty)$ with $\Phi_N(b_1,b_2)=0$ for some $N\in\bN$. After possibly removing finitely many points from $C_0$, namely, those belonging to the special hyperplanes defined by setting a coordinate to $0$ or $1728$, we can assume that $C_0$ is contained in $S^3_K$. Define $\cE_{i,0}$ and $\cE_i$ as in Sections \ref{sec:ecs} and \ref{sec:cover}, respectively. 

Thus, after possibly replacing $K$ with a finite extension and $C$ with a finite \'etale cover, we obtain a covering datum
\[\cD:=(\fG_1,\fG_2,\fG_3,{\fC}',{\fC},\tilde{C},\tilde{C}_0,\nu,x,s_1,\dots,s_\ell,\sigma_1,\dots,\sigma_\ell)\]
for $(C,s_0)$ (as defined in Section~\ref{sec:cover}).
Fix a Weil height $h$ on $\tilde{C}$. By standard properties of heights, it suffices for us to show that there exist constants $\newC{main-Weil-ht-mult}$ and ${\newC{main-Weil-ht-exp}}$ such that, for any modular point $s:=(x_1,x_2,x_3) \in C_0$, and any $\tilde{s}\in\nu^{-1}(s)\in\tilde{C}_0$, we have 
\[h(\tilde{s})\leq\refC{main-Weil-ht-mult}[K(s):K]^{\refC{main-Weil-ht-exp}}.\]

\subsubsection{Pullback representatives}\label{sec:pullback-reps}

For $I:=\{1,2\}$, let $\Lambda\subset \{1,\ldots,\ell\}\times I$ be a set of pullback representatives for $(\cD,I)$. As in \cite[Rem. 5.3]{Y(1)}, we have $[k,1]\neq [k,2]$ for all $k\in\{1,\ldots,\ell\}$. Define $\cE_\lambda$ and $\fG_\lambda$ as in Section \ref{sec:reps} and let ${\fC}_0\subset{\fC}'$ denote a Zariski open subset such that $s_1\in\fC_{0}(K)$ and $\fG_\lambda\times_{{\fC}'}{\fC}_0$ is an elliptic scheme for all $\lambda\in\Lambda$. Relabel $C:=\tilde{C}$ and $C_0:=\tilde{C}_0$, and replace ${C}_0$ with ${\fC}_{0,K}$. 
Write $\cE_{\lambda,0}:=\cE_{\lambda}\times_C C_0$. After possibly replacing $K$ with a finite extension and possibly removing finitely many points from $C_0$, we may assume that $H^1_{DR}(\cE_{\lambda,0}/C_0)$ is free for all $\lambda\in\Lambda$.

\subsubsection{Isogeny data}\label{sec:isogeny-data}

By \cite[Lemma 5.4]{Y(1)} and the hypothesis that $(b_1,b_2)$ is modular, for each $k\in\{1,\ldots,\ell\}$, there exist isogenies
\begin{equation*} 
\cE_{[k,1],s_1}\to(\sigma^*_k\cE_1)_{s_1}=\cE_{1,s_k} = \cE_{b_1} \to \cE_{b_2} = \cE_{2,s_k}=(\sigma^*_k\cE_2)_{s_1}\to\cE_{[k,2],s_1}
\end{equation*}
of elliptic curves over~$\ov K$. We want to choose isogenies 
\(\cE_{[k,1],s_1}\to\cE_{[k,2],s_1}\),
but we shall require certain compatibility conditions between these isogenies for different~$k$.
We achieve this using the following construction.

Consider the undirected graph $G$ with vertex set $V(G)=\Lambda$ and edge set $E(G)$ comprising pairs of vertices of the form $[k,1]$ and $[k,2]$ for some $k\in\{1,\ldots,\ell\}$. Note that, for each $e\in E(G)$, the index $k\in\{1,\ldots,\ell\}$ giving rise to that edge is not necessarily unique. Thus, for each $e\in E(G)$, we assign such a $k:=k(e)$, and we define $e(i):=[k(e),i]$ (which is an endpoint of~$e$).

Choose an isogeny $f_e \colon \cE_{e(1),s_1} \to \cE_{e(2),s_1}$ for each edge~$e$ as follows.
\begin{enumerate}
    \item Let $T$ denote a spanning forest for $G$. That is, for each connected component of $G$, choose a spanning tree, and let $T$ be the union of these spanning trees. 
    \item For each $e\in E(T)$, choose $f_{e}:\cE_{e(1),s_1}\to\cE_{e(2),s_1}$ arbitrarily.
    \item For each remaining $e\in E(G)$, there exists a unique sequence of edges $e_1,\ldots,e_r\in E(T)$ connecting $e(1)$ to $e(2)$; to each edge $e_i$ define
    \[   
g_{e_i} = 
     \begin{cases}
       f_{e_i} &\quad\text{if we traverse $e_i$ from $e_i(1)$ to $e_i(2)$}\\
       f^\vee_{e_i} &\quad\text{if we traverse $e_i$ from $e_i(2)$ to $e_i(1)$},
     \end{cases}
\]
and define $f_e:\cE_{e(1),s_1}\to\cE_{e(2),s_1}$ to be the composition $g_{e_r}\circ\cdots\circ g_{e_1}$.
\end{enumerate}

After possibly replacing $K$ with a finite extension, we may assume that the $f_{e}$ are defined over~$K$. For each $\lambda\in\Lambda$, we choose bases $\{\omega_{\lambda},\eta_{\lambda}\}$ for $H^1_{DR}(\cE_{\lambda,0}/C_0)$ such that, for any $e\in E(G)$, we have
\begin{equation}\label{eqn:basis-compactibility}
\omega_{e(1),s_1} = \mu_ef_{e}^*\omega_{e(2),s_1} \text{ and } \eta_{e(1),s_1} = \mu_e f_{e}^*\omega_{e(2),s_1},
\end{equation} 
for some $\mu_e\in\bQ$ (this is the analogue of condition \eqref{eqn:omega-f0} of Section \ref{sec:main}, with $s_1\in C(K)$ playing the role of $s_0$ and $f_{e}$ the role of $f_0$). In order to do this, consider a connected component $G' \subset G$.
Choose a vertex $\lambda_0 \in V(G')$ and choose an arbitrary basis $\{\omega_{\lambda_0}, \eta_{\lambda_0} \}$ for $H^1_{DR}(\cE_{\lambda_0,0}/C_0)$.
Working outwards from~$\lambda_0$, following the edges of the tree $G' \cap T$, we then choose a basis for $H^1_{DR}(\cE_{\lambda,0}/C_0)$ for each $\lambda \in V(G')$, satisfying~\eqref{eqn:basis-compactibility} for all $e \in E(G' \cap T)$ (for $e \in E(G' \cap T)$, we may choose $\mu_e$ arbitarily; for example, we may choose $\mu_e = 1$ for all $e \in E(G' \cap T)$).
Finally, observe that the chosen bases satisfy~\eqref{eqn:basis-compactibility} for all $e \in E(G') \setminus E(T)$, due to the construction of the $f_e$ above.
(Note that we are no longer free to choose the values of~$\mu_e$ for $e \in E(G') \setminus E(T)$, and we might not have $\mu_e=1$ for all such~$e$ because any edge~$e_i$ which was traversed in the reverse direction while constructing~$f_e$ will introduce a factor of $\deg(f_{e_i})$ in~$\mu_e$.)


Let $F_{\lambda mn}\in K[\![X]\!]$ ($1\leq m,n\leq 2$) denote the G-functions associated with $\cE_{\lambda,0}\to C_0$, $x$, $s_1$ and $\{\omega_{\lambda},\eta_{\lambda}\}$ (see Section \ref{sec:GF-ab-sch}) and define
\[\cG=\{F_{\lambda mn}:\lambda\in\Lambda,\ 1\leq m,n\leq 2\}.\]
Let $\cS:=((r_v,U_{v,1}, \dotsc, U_{v,\ell})_v)$ denote a neighbourhood system for $(C,x)$ and let $H$ denote an auxillary G-function for $\cS$. Define $\cG_H:=\cG\cup\{H\}$. For each $v\in\Sigma_K$, define
\[R_v:=\min\{1,R(H^\van),R(F^\van):F\in\cG\}\leq r_v.\]
For each $k\in\{1,\ldots,\ell\}$, let $U_{R_v,k}=(x^\van|_{U_{v,k}})^{-1}(D(0,R_v,K_v))$, and let $\Delta_v:=U_{R_v,1}$.
Then, for each $e\in E(G)$, 
\[\{C,\fC,\fG_{e(i)},C_0,\cE_{e(i),0},\fC_0,s_1,f_{e},\{\omega_{e(i)},\eta_{e(i)}\},x,F_{e(i)mn},R_v,\Delta_v:1\leq i,m,n\leq 2, v\in\Sigma_K\}\]
is an isogeny datum (see Section~\ref{sec:main}).


\subsubsection{Constructing relations}\label{sec:construct}

Consider $s\in C_0$ with $\nu(s)$ modular and replace $K$ with $K(s)$.
There exists an isogeny $\cE_{1,s} \to \cE_{2,s}$.
(By the definition of modular points, there also exists an isogeny $\cE_{1,s} \to \cE_{3,s}$, but the latter isogeny is not required for this step of the proof; it will be used in Section~\ref{sec:degree} below.)
Hence, for each $k\in\{1,\ldots,\ell\}$, there is an isogeny
\[\cE_{[k,1],\sigma_k^{-1}(s)}\to(\sigma^*_k\cE_1)_{\sigma^{-1}_k(s)}=\cE_{1,s}\to\cE_{2,s}=(\sigma^*_k\cE_2)_{\sigma^{-1}_k(s)}\to\cE_{[k,2],\sigma_k^{-1}(s)},\]
the outer isogenies being afforded to us, again, by \cite[Lemma 5.4]{Y(1)}. After replacing $K$ with an extension of bounded degree, we can and do assume these are $K$-isogenies.

Now, for any $v\in\Sigma_{K}$ such that $|x(s)|_v<R_v$, we have $s^\van\in U_{R_v,k}$ for some $k\in\{1,\ldots,\ell\}$ (which may depend on~$v$, of course). As in \cite[Rem. 5.6]{Y(1)}, we deduce that $\sigma^{-1}_k(s)^{\van}\in U_{R_v,1}=\Delta_v$. 
Let $e\in E(G)$ be such that $e(i)=[k,i]$ for $i\in\{1,2\}$ (such an $e$ exists by the definition of $G$). Hence, by Proposition \ref{prop:rel} applied to
\[\{C,\fC,\fG_{e(i)},C_0,\cE_{e(i),0},\fC_0,s_1,f_{e},\{\omega_{e(i)},\eta_{e(i)}\},x,F_{e(i)mn},R_v,\Delta_v:1\leq i,m,n\leq 2, v\in\Sigma_K\},\]
there exists a $v$-adic relation $Q_v$ of degree at most $2$ between the evaluations at $x(\sigma^{-1}_k(s))=x(s)$ of the $F_{e(i)mn}$ for $ i,m,n\in\{1,2\}$, with the property that $Q_v$ is not contained in the ideal $I$ of $\Qbar[X_{imn}:1\leq i,m,n\leq 2]$ generated by the elements
\[\det((X_{1mn})_{mn})-\det((X_{2mn})_{mn}).\] 

Let $\Sigma_s$ denote the set of $v\in\Sigma_K$ satisfying $|x(s)|_v<R_v$ and let 
\begin{equation*}
    Q:=\prod_{v\in\Sigma_s}Q_v.
\end{equation*}
Then $Q$ is a global relation between the evaluations at $x(s)$ of the elements of $\cG_H$, thanks to the definition of~$R_v$.

\subsubsection{Constraining the degree of~$Q$} \label{sec:degree}

Consider a place $v \in \Sigma_s \cap \Sigma_{K,f}$ and the corresponding index~$k$ such that $s^{\van}\in U_{R_v,k}$.
Write $\fs_k$ and~$\fs$ for the $\cO_K$-extensions of $s_k$ and~$s$, respectively, in $\fC'$.
Since $s$ is modular, there exists an isogeny
$\cE_{1,s} \to \cE_{3,s}$
and hence an isogeny between the Néron models $\fG_{1,\fs} \to \fG_{3,\fs}$.
Since $s^{\van}\in U_{R_v,k}$ and $U_{R_v,k}$ is a simple neighbourhood, we have $\fs(\fp_v)=\fs_k(\fp_v)$.
Therefore, we have an isogeny
\[\fG_{1,\fs_k,k_v} \cong \fG_{1,\fs,k_v} \to \fG_{3,\fs,k_v} \cong \fG_{3,\fs_k,k_v}=\bG_{m,k_v}.\]
Hence, $\cE_{b_1}=\cE_{1,s_k}$ has bad reduction at~$v$.

In other words, $v\in\Sigma_s$ implies that either $v\in\Sigma_{K,\infty}$ or $v$ is a (finite) place of bad reduction for $\cE_{b_1}$.
Therefore, the degree of $Q$ is bounded by $\newC{main-deg-mult}[K:\bQ]$ for some constant $\refC{main-deg-mult}$ independent of~$s$. Indeed, the number of places of bad reduction for $\cE_{b_1}$ over~$[\bQ(s_0):\bQ]$ is independent of~$s$. Hence, the number of places of bad reduction for $\cE_{b_1}$ over~$K$ is bounded by a constant multiple of~$[K:\bQ]$.
The number of places in $\Sigma_{K,\infty}$ is also bounded by $[K:\bQ]$.

We conclude that Theorem \ref{teo:main-ht} follows from \cite[VII, Thm.~5.2]{And89} (cf.\ \cite[Thm.~2.8]{Y(1)}) if we can show that $Q$ is nontrivial. This will be the subject of the next section.

\subsubsection{Nontriviality}\label{sec:non-triviality}
By \cite[Lemma 2.9]{Y(1)}, it suffices to show that $Q$ is a non-trivial relation between the elements of $\cG$ (as opposed to $\cG_H$). Therefore, let $J$ denote the (homogeneous in $X_{\lambda mn}$) ideal of
\[\Qbar[X][X_{\lambda mn}:\lambda\in\Lambda,\ 1\leq m,n\leq 2\}]\] 
 comprising the functional relations between the elements of $\cG$. We need to show that $Q$ does not belong to the specialisation of $J$ at $x(s)$.

 \begin{lemma}
     The ideal $J$ is prime.
 \end{lemma}

 \begin{proof}
     By definition, $J$ is the homogeneous part of the kernel of the homomorphism
\[\Qbar[X][X_{\lambda mn}:\lambda\in\Lambda,\ 1\leq m,n\leq 2\}]\to\Qbar[\![X]\!]\]
defined by $X_{\lambda mn}\mapsto F_{\lambda mn}$. Since the latter is clearly prime, so too is $J$ by \cite[00JM, Lemma 10.57.7]{stacks}.
 \end{proof}

To ease notation, we set $r:=|\Lambda|$. Choose $v\in\Sigma_{K,\infty}$. The $F_{\lambda mn}$ give rise to a holomorphic function $D(0,R_v,\bC)\to\gSL^{r}_2(\bC)$, and we let $\Gamma$ denote the graph of this function. 

\begin{lemma}\label{graph-dense}
    The set $\Gamma$ is $\bC$-Zariski dense in $\bA^1\times\gSL_2^r$. 
\end{lemma} 

\begin{proof}
Denote by $\Omega_v:\Delta_v\to\rM_{2g}(\bC)$ the function giving rise to the $F_{\lambda mn}$, as in Section~\ref{sec:GF-ab-sch}, for some choice of horizontal basis of $H^1_{DR}(\cE_\lambda/C)^\van(\Delta_v)$.
Write $\Omega_{v\lambda mn}$ for its components.

By the Legendre period relation, $\det\Omega_v$ is constant.
It follows that $\Gamma$ is indeed contained in $\bA^1\times\gSL_2^r$.

Therefore, it suffices to show that
\[\trdeg_{\bC(X)}\bC(X)(F_{\lambda mn}:\lambda\in\Lambda,\ 1\leq m,n\leq 2)\geq \dim(\gSL_2^r) = 3r.\]
This is clearly implied by 
\[\trdeg_{\bC(C)}\bC(C)(\Omega_{v\lambda mn}:\lambda\in\Lambda,\ 1\leq m,n\leq 2)\geq 3r.\]
 
As in Remark \ref{rem:periods}, if we choose suitable horizontal bases for $H^1_{DR}(\cE_\lambda/C)^\van(\Delta_v)$, then the $\Omega_{v\lambda mn}$ are the entries of a period matrix for $\cE_\lambda\to C$. Therefore, the geometric Andr\'e--Grothendieck period conjecture, proved in this case by Ayoub \cite{Ayoub:GPC} and Nori (unpublished)---see \cite[Thm.~1.1]{BT:GPC} and the remarks thereafter---yields exactly this inequality. Indeed, by \cite[Cor.~4.17]{BT:GPC} the relative motivic Galois group $\gG_{\rm mot}(M/\bC)$ is always equal to the Zariski closure of the monodromy of the Betti local system of~$M$. In our case, the latter is $\gSL_2^r$ because the $\cE_\lambda$ are not generically CM or generically isogenous, so $\dim \gG_{\rm mot}(M/\bC) = 3r$.
\end{proof}

By definition, $V(J)$ is the cone over the Zariski closure of~$\Gamma$. Thus,  Lemma~\ref{graph-dense} implies that 
\[V(J)=\bA^1\times\{(g_1,\ldots,g_r)\in M_2(\bC)^r:\det(g_1)=\ldots=\det(g_{r})\}.\]
Since $J$ is prime and therefore radical, we deduce that $J$ is the ideal of
\[ \Qbar[X][X_{\lambda mn} : \lambda\in\Lambda,\ 1\leq m,n\leq 2] \]
generated by the elements 
\begin{equation*}
    \det((X_{\lambda_1 mn})_{mn})-\det((X_{\lambda_2 mn})_{mn}) 
\end{equation*}
for $\lambda_1,\lambda_2\in\Lambda$. The specialisation $J_0$ of $J$ at $x(s)$ is therefore the ideal of 
\[\Qbar[X_{\lambda mn}:\lambda\in\Lambda,\ 1\leq m,n\leq 2]\] 
generated by these elements. 

By faithful flatness of a polynomial ring over its ring of coefficients, we conclude that $J_0 = J \cap \Qbar[X_{\lambda mn}:\lambda\in\Lambda,\ 1\leq m,n\leq 2]$, hence $J_0$ is also prime. Therefore, it suffices to show that $Q_v\notin J_0$ for all $v\in\Sigma_s$.
This follows from the fact that each $Q_v$ is not contained in the ideal $I$ of Proposition \ref{prop:rel} (again by faithful flatness of a polynomial ring over its ring of coefficients).
\qed

\subsection{Proof of Theorem \ref{teo:super-ht}}
\label{proof:super-ht}

Consider the situation in Theorem \ref{teo:super-ht}. In particular, let $s_0:=(b_1,b_2,b_3)\in C$ be the modular point referred to in the statement, and let $K$ denote a number field over which $C$ and $s_0$ are defined. Again, to match the notation of the preceding sections, we relabel $C_0:=C$ and $C:=\overline{C}$.

\subsubsection{Reductions}

After removing at most finitely many points from $C_0$, we can assume that $C_0$ is contained in $S^3_K$. Thus, after possibly replacing $K$ with a finite extension and $C$ with a finite \'etale cover, we obtain a covering datum \[\cD:=(\fG_1,\fG_2,\fG_3,\fC',{\fC},\tilde{C},\tilde{C}_0,\nu,x,s_1,\dots,s_\ell,\sigma_1,\dots,\sigma_\ell)\] for $(C,s_0)$, and we fix a Weil height $h$ on $\tilde{C}$. As before, it suffices for us to show that there exist constants $\newC{super-Weil-ht-mult}$, ${\newC{super-Weil-ht-exp}}$ and ${\newC{super-Weil-ht-exp2}}$ such that, for any modular point $s:=(x_1,x_2,x_3)\in C$ of supersingular exponent $\delta$, and any $\tilde{s}\in\nu^{-1}(s)\subset\tilde{C}_0$, we have 
\[h(\tilde{s})\leq\refC{super-Weil-ht-mult}[K(s):K]^{\refC{super-Weil-ht-exp}}\Delta(s)^{{\refC{super-Weil-ht-exp2}}\delta}.\]

\subsubsection{Setup}
We copy the notations and repeat the constructions of Section \ref{sec:pullback-reps}, this time with $I:=\{1,2,3\}$. We construct a graph $G$ as in Section \ref{sec:isogeny-data}, but this time with edges between $[k,j]$ and $[k,j+1]$ for $j\in\{1,2\}$. Thus, to each $e\in E(G)$, we assign this $j:=j(e)$ and, then, $e(i):=[k(e),j(e)+(i-1)]$. Copying the rest of Section \ref{sec:isogeny-data}, we obtain an isogeny datum
\[\{C,\fC,\fG_{e(i)},C_0,\cE_{e(i),0},\fC_0,s_1,f_{e},\{\omega_{e(i)},\eta_{e(i)}\},x,F_{e(i)mn},R_v,\Delta_v:1\leq i,m,n\leq 2, v\in\Sigma_K\}\]
for each $e\in E(G)$.

\subsubsection{Constructing relations}

Let $s\in C_0$ with $\nu(s)=(x_1,x_2,x_3)$ modular of supersingular exponent $\delta$. Replace $K$ with $K(s)$ and write $\fs_k$ and~$\fs$ for the $\cO_K$-extensions of $s_k$ and $s$, respectively, in $\fC'$.
Note that, for any $v\in\Sigma_{K}$, the condition $|x(s)|_v<R_v$ implies $s^\van\in U_{R_v,k}$ for some $k\in\{1,\ldots,\ell\}$. If $v\in\Sigma_{K,f}$, then, since $U_{R_v,k}$ is a simple neighbourbood, we have
\[\fG_{i,\fs,k_v} \cong \fG_{i,\fs_k,k_v} \]
for all $i\in I$, which implies that the $j$-invariants $x_i$ and $b_i$ are congruent modulo the maximal ideal of $\cO_{K_v}$, which is to say, $|x_i-b_i|_v<1$.

Define $\Sigma_{s}\subset\Sigma_{K}$ to be the set of $v\in\Sigma_K$ satisfying $|x(s)|_v<R_v$. In particular, $\Sigma_s\cap\Sigma_{K,f}$ is contained in the set of $v\in\Sigma_{K,f}$ for which \eqref{eqn:supersing-close-condition} holds. By repeating the arguments of Section \ref{proof:main-ht}, invoking Proposition \ref{prop:rel2} in the place of Proposition \ref{prop:rel}, we obtain a global relation between the evaluations at $x(s)$ of the elements of $\cG_H$, which is equal to the product of a relation of degree at most $4$ (dealing with the good ordinary places) and a relation
\begin{equation*}
    Q^{\rm nord}:=\prod_{v\in\Sigma^{\rm nord}_{s}}Q_v
\end{equation*}
where each $Q_v$ is a relation of degree at most $2$ and $\Sigma^{\rm nord}_{s}\subset\Sigma_{s}$ is the complement in $\Sigma_{s}$ of the set of finite places for which the $\cE_{i,s}$ have good ordinary reduction. The fact that $Q$ is nontrivial is proved exactly as in Section \ref{sec:non-triviality}, hence, we omit it.

\subsubsection{Constraining the degree}
\label{sec:degree-super-ht}

It follows that there exists a constant $\newC{super-deg-mult}$ independent of $s$ such that the degree of $Q$ is bounded by $\refC{super-deg-mult}([K:\bQ]+|\Sigma^{\rm sup}_{s}|)$, where $\Sigma^{\rm sup}_{s}$ is the set of finite places in $\Sigma_{s}$ for which the elliptic curve with $j$-invariant $b_1$ has supersingular reduction. By assumption, we have
\[|\Sigma^{\rm sup}_{s}|\leq [K:\bQ]\Delta(s)^{\delta}\]
and so the result follows from \cite[Thm. 5.2]{And89}. 
\qed

\begin{remark}\label{rem:constant}
   From \cite[VII, Thm.~5.2]{And89}, noting that the $\Qbar(X)$-linear span of~$\cG_H$ is closed under differentiation, we see that Theorem \ref{teo:super-ht} holds for any 
   \[\refC{super-ht-exp2}>3(|\cG_H|-1)=3(4|\Lambda|+1-1)=12|\Lambda|.\]
   In particular, since $|\Lambda|\leq 3\ell$, we deduce that Theorem \ref{teo:super-ht} holds for any $\refC{super-ht-exp2}>36\ell$.
\end{remark}

\section{Proving Theorems \ref{teo:main} and \ref{teo:super}}\label{sec:derivations}

The proofs of Theorems \ref{teo:main} and \ref{teo:super} follow the Pila--Zannier strategy. In fact, following word-for-word the proof of \cite[Prop. 5.1]{HP12}, it suffices to obtain the following large Galois orbit results, which serve as substitutes for \cite[Lemma 4.2]{HP12}.

\begin{theorem}\label{teo:galois-main}
    Let $C$ satisfy the conditions of Theorem \ref{teo:main}. There exist constants $\newC{main-galois-mult},\newC{main-galois-exp1}>0$ such that, for any modular point $s=(x_1,x_2,x_3)\in C$, we have
    \[ [\bQ(x_1,x_2,x_3):\bQ]\geq\refC{main-galois-mult}\Delta(s)^{\refC{main-galois-exp1}}.\]
\end{theorem}

\begin{theorem}\label{teo:galois-super}
    Let $C$ and $(b_1,b_2,b_3)$ be as in Theorem \ref{teo:super} and let $\refC{super-ht-exp2}$ be the constant afforded to us by Theorem \ref{teo:super-ht}. For any $\delta<(2\refC{super-ht-exp2})^{-1}$ there exist constants $\newC{super-galois-mult},\newC{super-galois-exp1}>0$ such that, for any modular point $s=(x_1,x_2,x_3)\in C$ of supersingular exponent $\delta$, we have
    \[ [\bQ(x_1,x_2,x_3):\bQ]\geq\refC{super-galois-mult}\Delta(s)^{\refC{super-galois-exp1}}.\]
\end{theorem}

The proofs of these theorems follow almost exactly the proof of \cite[Prop.~5.15]{Y(1)}, which is essentially an extract from the proof of \cite[Lemma~4.2]{HP12}. The key input is a Masser--W\"ustholz-type isogeny estimate. For Theorem~\ref{teo:galois-super}, we require the version of Gaudron--R\'emond \cite[Thm.~1.4]{GR14EC} (it is the exponent~$2$ in \cite[Thm.~1.4]{GR14EC} which leads to the condition $\delta<(2\refC{super-ht-exp2})^{-1}$ in \cref{teo:galois-super}).

\appendix
\section{Rigid analytic neighbourhoods}

We establish a version of \cite[Lemma~5.5]{Y(1)}, written in terms of rigid geometry instead of only $\bC_p$-points, and thereby give a rigid analytic proof of \cite[Lemma~5.5]{Y(1)}.
Along the way, we establish a version which talks about the neighbourhood of a single point where $x$ has a simple zero, instead of requiring all zeroes of~$x$ to be simple.
For convenience, we also prove that the neighbourhoods we construct can be chosen to be simple, as defined in Section~\ref{sec:simp-n}.

First we check that an \'etale morphism of rigid spaces restricts to an isomorphism on some open neighbourhood of each $K$-point.
(Note that it is not true in general that \'etale morphisms are local isomorphisms---the open neighbourhoods on which they are isomorphisms need not form an admissible covering.)

\begin{lemma} \label{rigid-etale-local-isomorphism}
Let $K$ be a non-archimedean field of characteristic zero.
Let $X$ be a quasi-separated rigid space over~$K$.
Let $f \colon X \to \bar D(0,r,K)$ be a finite \'etale morphism from $X$ to the closed disc of radius $r$.
Suppose there is a $K$-point $x_0 \in X(K)$ such that $f(x_0)=0$.
Then there exists an open subspace $Z \subset X$ containing~$x_0$ and a positive real number $r' \leq r$ such that $f|_Z \colon Z \to \bar D(0,r',K)$ is an isomorphism of rigid spaces.
\end{lemma}

\begin{proof}
Since $f$ is \'etale, it is flat, so its image is open in $D(0,r,K)$.
In particular, the image of~$f$ contains some disc $D(0,R,K)$.
So, after replacing $X$ by $f^{-1}(D(0,R,K))$ and $r$ by~$R$, we may assume that $f$ is surjective (i.e.\ it is a finite \'etale covering).

Let $L$ denote the completion of the algebraic closure of~$K$.
By \cite[Thm.~2.1]{Lut93}, after extending scalars to~$L$, the finite \'etale covering $f_L \colon X_L \to \bar D(0,r,L)$ splits over~$\bar D(0,r',L)$ for some $r'$ with $0 < r' \leq r$.
In other words, letting $Y = f^{-1}(\bar D(0,r,k))$, we have that $Y_L$ is the disjoint union of subspaces $Y_1, \dotsc, Y_\ell$ such that $f_L|_{Y_i} \colon Y_i \to \bar D(0,r',L)$ is an isomorphism for each~$i$.

Let $Z$ denote the connected component of~$Y$ which contains~$x_0$.
Since $Z$ is connected and it has a $K$-point~$x_0$, $Z_L$ is connected \cite[Thm.~3.2.1]{Con99}.
Hence $Z_L$ is a connected component of~$Y_L$, that is, $Z_L = Y_i$ for some~$i$.
Thus $f_L|_{Z_L} \colon Z_L \to \bar D(0,r',L)$ is an isomorphism, so $f|_Z \colon Z \to \bar D(0,r',k)$ is an isomorphism \cite[Thm.~A.2.4]{Con06}.
\end{proof}

Henceforth, let $C$ denote a smooth projective geometrically irreducible algebraic curve over a number field~$K$, and let $x \in K(C)$ be a rational function. Let $\fC$ be a regular projective $\cO_K$-model of $C$.

\begin{proposition} \label{rigid-coordinate-disc}

Suppose that $s_0 \in C(K)$ is a simple zero of~$x$.

For each $v\in\Sigma_{K,f}$, there exists a real number $r_v > 0$ and a simple open subspace $U_v \subset C^\van$ with the following properties:
\begin{enumerate}[(i)]
\item $r_v \geq 1$ for almost all $v\in\Sigma_{K,f}$;
\item $s_0^\van \in U_v$;
\item for each $v\in\Sigma_{K,f}$,  the morphism $x^\van$ restricts to an isomorphism of rigid spaces from $U_v$ to the rigid open disc $D(0,r_v,K_v)$;
\item for every rational function $f \in K(C)$ which is regular at~$s_0$, if $\hat f \in \powerseries{K}{X}$ denotes the Taylor series of~$f$ around~$s_0$ in terms of the local parameter~$x$, then, for all $s \in U_v$ satisfying $\abs{x^\van(s)} < R(\hat f^\van)$, we have $\hat f^\van(x^\van(s)) = f^\van(s)$.
\end{enumerate}
\end{proposition}

\begin{proof}
Since $x \colon C \to \bP^1_K$ is a non-constant morphism between integral curves, it is flat.
By hypothesis, $x$ is unramified at~$s_0$.
Thus, $x$ is étale at~$s_0$.

By \cite[Lem. 6.2]{LGO}, there exists a regular projective $\cO_K$-model $\fC'$ and a morphism $f:\fC'\to\fC$ such that $x$ extends to a morphism $\xi:\fC'\to\bP^1_{\cO_K}$. Since $\fC' \to \Spec(\cO_K)$ is flat (by the definition of an~$\cO_K$-model), $\xi$ is étale at~$s_0$ by \cite[Rmk.~17.8.3]{EGAIV4}.
It follows that $\xi$ is étale on a non-empty Zariski open neighbourhood of~$s_0$ in $\fC'$. 
Hence, there exists a dense open subset $\fV \subset \Spec(\cO_K)$ such that $\xi$ is étale on $\fs_0(\fV)$, where $\fs_0 \colon \Spec(\cO_K) \to \fC'$ is the $\cO_K$-extension of $s_0$.

\smallskip

We shall give two constructions of simple open subspaces $U_v \subset C^\van$ satisfying (ii) and (iii), depending on whether $\fp_v$ lies in~$|\fV|$ or not.
The construction for primes in~$|\fV|$ gives $r_v=1$, while the construction for other primes may give $r_v < 1$.
This is sufficient to ensure that (i) holds.

\smallskip


First consider $v\in\Sigma_{K,f}$ such that $\fp:=\fp_v \in |\fV|$.
Let $\fC'_\fp$ denote the fibre above~$\fp:=\fp_v$ of $\fC' \to \Spec(\cO_K)$.
Let $\hat{\fC'}_{\fC'_\fp}$ and $\hat{\fC'}_{\fs_0(\fp)}$ denote the formal completions of $\fC'$ along $\fC'_\fp$ and $\fs_0(\fp)$ respectively.
Define $\hat\fC_{\fC_\fp}$ analogously.
Since $\fC' \to \Spec(\cO_K)$ and $\fC \to \Spec(\cO_K)$ are proper, the open immersions $\hat{\fC'}_{\fC'_\fp}^\rig \to \hat{\fC}_{\fC_\fp} \to C^\van$ are isomorphisms.

Let $U_v:=(\hat{\fC}')_{\fs_0(\fp)}^\rig$. By \cite[Prop.~0.2.7]{Berthelot96}, $U_v$ is a simple open subspace of $C^\van$ with respect to $\fC'$. The existence of the following commutative diagram implies that $U_v$ is a simple open subspace of $C^\van$ with respect to $\fC$: 
\[ \begin{tikzcd}
    U                  \ar[r, hook]
  & (\hat{\fC}')_{\fC'_{k_v}}^\rig \ar[d, "\red'_{v}"] \ar[r, "\cong"]
  & \hat{\fC}_{\fC_{k_v}}^\rig \ar[d, "\red_{v}"]
\\& \fC'_{k_v}        \ar[r]
  & \fC_{k_v}
\end{tikzcd} \]
Indeed, it is immediate that $U_v$ satisfies properties (1) and (2) in the definition of simple open subspace with respect to~$\fC$.
Property (3) follows from the fact that the value of $\red_v$ on $U$ is equal to the image of $f(\fs_0(\fp))$ in $\fC_{k_v}$, which is a smooth point by \cite[Exercise~4.3.25(c)]{Liu06}.

Let $\hat{\bP}^1_{\fz_0(\fp)}$ denote the formal completion of $\bP^1_{\cO_K}$ along $ \fz_0(\fp) $, where $\fz_0 \colon \Spec(\cO_K) \to \bP^1_{\cO_K}$ denotes the zero section.
Since $\xi$ is étale at $\fs_0(\fp)$, it induces an isomorphism of formal schemes $\hat{\xi}_{\fs_0(\fp)} \colon \hat{\fC}'_{\fs_0(\fp)} \to \hat{\bP}^1_{\fz_0(\fp)}$.
Applying Berthelot's rigid generic fibre functor, we obtain an isomorphism of rigid spaces $\hat{\xi}_{\fs_0(\fp)}^\rig \colon U_v=(\hat{\fC}')_{\fs_0(\fp)}^\rig \to \hat{\bP}^{1,\rig}_{\fz_0(\fp)}$.
Since $\hat{\bP}^{1,\rig}_{\fz_0(\fp)} = \Spf(\powerseries{\cO_{K,\fp}}{X}, (\fp, X))^\rig$ is the rigid open unit disc, we conclude that $U_v$ satisfies (iii) with $r_v=1$.
Clearly, $U_v$ satisfies~(ii).

\smallskip

Now consider $v\in\Sigma_{K,f}$ such that $\fp_v\notin|\fV|$ (this construction applies for all $v\in\Sigma_{K,f}$, but we use the previous construction for $v$ such that $\fp_v\in|\fV|$).
Since $x$ is finite and étale on a Zariski open neighbourhood of~$s_0$ in~$C$, there is a rigid open neighbourhood $V_v$ of~$s^\van_0$ in~$C^\van$ such that $x^\van|_{V_v}$ is finite étale.
After shrinking $V_v$, we may assume that the image of $V_v$ is contained in a disc around~$0$.
By \cref{rigid-etale-local-isomorphism}, there exists an open subspace $U_v \subset V_v$ containing~$s_0^\van$ such that $x^\van|_{U_v}$ is an isomorphism of rigid spaces from~$U_v$ to a closed rigid disc $\bar D(0,r_v,K_v)$.
After shrinking~$U_v$ and~$r_v$, we may arrange that $U_v$ is simple and that $x^\van|_{U_v}$ is instead an isomorphism onto the open rigid disc $D(0,r_v,K_v)$.
Thus, $U_v$ satisfies (ii) and~(iii). 

\smallskip

Finally we verify (iv) for any $U_v$ satisfying (ii) and~(iii).
Since the analytification functor induces a morphism of locally G-ringed spaces $(C^\van, \cO_{C^\van}) \to (C_{K_v}, \cO_{C_{K_v}})$, the (algebraic) Taylor series of~$f$ with respect to~$x$ is the same as the (rigid analytic) Taylor series of $f^\van$ with respect to $x^\van$.
Since the Taylor series of a rigid function~$g$ on a disc converges to~$g$ inside its radius of convergence, (iv) holds.
\end{proof}

Now we prove a rigid analogue of \cite[Lemma~5.5]{Y(1)}.

\begin{proposition} \label{rigid-discs}
Suppose that $x$ has $\ell$ distinct unramified zeros $s_1, \dotsc, s_\ell \in C(K)$.

For each $v\in\Sigma_{K,f}$, there exists a real number $r_v > 0$ and simple open subspaces $U_{v,1}, \dotsc, U_{v,\ell} \subset C^\van$ with the following properties:
\begin{enumerate}[(i)]
\item $r_v \geq 1$ for almost all $v\in\Sigma_{K,f}$;
\item $s_k^\van \in U_{v,k}$;
\item for each $v\in\Sigma_{K,f}$, the spaces $U_{v,1}, \dotsc, U_{v,\ell}$ are pairwise disjoint, and form an admissible covering of the preimage under $x^\van$ of the rigid open disc $D(0,r_v,K_v)$;
\item for each $v\in\Sigma_{K,f}$ and~$k=1,\ldots,\ell$, the morphism $x^\van$ restricts to an isomorphism of rigid spaces from $U_{v,k}$ to the rigid open disc $D(0,r_v,K_v)$;
\item for every rational function $f \in K(C)$ which is regular at~$s_k$, if $\hat f \in \powerseries{K}{X}$ denotes the Taylor series of~$f$ around~$s_k$ in terms of the local parameter~$x$, then, for all $s \in U_{v,k}$ satisfying $\abs{x^\van(s)} < R(\hat f^\van)$, we have $\hat f^\van(x^\van(s)) = f^\van(s)$.
\end{enumerate}
\end{proposition}

\begin{proof}
Apply \cref{rigid-coordinate-disc} to each of~$s_1, \dotsc, s_\ell$, thereby yielding $r_{v,1}, \dotsc, r_{v,\ell}$ and $U_{v,1}, \dotsc, U_{v,\ell}$ for all~$v\in\Sigma_{K,f}$.
The latter automatically satisfy (ii), (iv) and~(v).
However, they need not satisfy (iii).

\smallskip

By \cite[Lem.~6.2]{LGO}, there exists a regular projective $\cO_K$-model $\fC'$ and a morphism $f:\fC'\to\fC$ such that $x$ extends to a morphism $\xi:\fC'\to\bP^1_{\cO_K}$. 
Let $\fs_1, \dotsc, \fs_k \colon \Spec(\cO_K) \to \fC'$ denote the $\cO_K$-extensions of $s_1,\ldots,s_\ell$.
As in the proof of Proposition \ref{rigid-coordinate-disc}, there is a dense open subset $\fV \subset \Spec(\cO_K)$ such that $\xi$ is \'etale on $\fs_k(\fV)$ for all $k=1, \dotsc, \ell$.
Since $s_1, \dotsc, s_k$ are pairwise distinct, after shrinking $\fV$, we may also assume that $\fs_j(\fp) \neq \fs_k(\fp)$ for all $\fp \in |\fV|$ and all $j \neq k$.

\smallskip

Now, suppose first that $\fp_v \in |\fV|$. 
As in the proof of Proposition \ref{rigid-coordinate-disc}, we can and do suppose that 
\[ U_{v,k} = (\hat{\fC'})_{\fs_k(\fp)}^\rig = (\red_v')^{-1}(\fs_k(\fp)) \]
and $r_v:=r_{v,k}=1$ for all $k=1,\ldots,\ell$.
Since $\fs_1(\fp), \dotsc, \fs_\ell(\fp)$ are distinct (by our choice of $\fp \in |\fV|$), we deduce that $U_{v,1}, \dotsc, U_{v,\ell}$ are pairwise disjoint.

Now consider the sets
\[ Z_{\fp,k} := \fC'_\fp \setminus \{ \fs_j(\fp) : 1 \leq j \leq \ell, \, j \neq k \} \]
for $1 \leq k \leq \ell$, which form a Zariski open covering of~$\fC'_\fp$.
Then the rigid spaces $(\red_v')^{-1}(Z_{\fp,1}), \dotsc, (\red_v')^{-1}(Z_{\fp,\ell})$ form an admissible covering of $C^\van$.
Since $U_{v,k} \subset (\red_v')^{-1}(Z_{\fp,k})$ and $U_{v,k}$ is disjoint from $(\red_v')^{-1}(Z_{\fp,j})$ for $j \neq k$, it follows that $U_{v,1}, \dotsc, U_{v,\ell}$ form an admissible covering of their union.

Next, suppose that $\fp_v \not\in |\fV|$. If necessary, we will shrink $U_{v,1}, \dotsc, U_{v,\ell}$ to make them disjoint and an admissible covering.
To that end, choose $r'_v < \min\{ r_{v,1}, \dotsc, r_{v,\ell} \}$ and let $U_{v,k}'$ denote the preimage of the closed disc $\bar D(0,r_v',K_v)$ inside $U_{v,k}$.
Then the rigid spaces $U_{v,k}'$ are quasi-compact, so $U_{v,j}' \cap U_{v,k}'$ are quasi-compact for each $j,k$.
Hence, if $U_{v,j}' \cap U_{v,k}' \neq \emptyset$, the absolute value of the function $x^\van$ has a minimum value~$R_{v,jk}$ on $U_{v,j}' \cap U_{v,k}'$.
If $j \neq k$, then $s_j^\van, s_k^\van \not\in U_{v,j}' \cap U_{v,k}'$, so $R_{v,jk} > 0$.
Let
\[ r_v = \min \{ r_v', \, R_{v,jk} : j \neq k, U_{v,j}' \cap U_{v,k}' \neq \emptyset \} > 0. \]
By construction, after replacing $U_{v,k}$ by the preimage of $D(0,r_v,K_v)$, the sets $U_{v,1}, \dotsc, U_{v,\ell}$ are pairwise disjoint.

Since the sets $U_{v,k}'$ are affinoid and there are finitely many of them, they form an admissible covering of their union.
After the final replacement of $U_{v,k}$, we have $U_{v,k} \subset U_{v,k}'$ and $U_{v,k}$ is disjoint from $U_{v,j}$ for $j \neq k$.
Hence $U_{v,1}, \dotsc, U_{v,\ell}$ form an admissible covering of their union.

\smallskip

To conclude, we note that every point of $(\bA^1)^\van$ has at most $\ell$ preimages under $x^\van$.  On the other hand, since $U_{v,1}, \dotsc, U_{v,\ell}$ are pairwise disjoint, every point of $D(0,r_v,K_v)$ has $\ell$ distinct preimages in $U_{v,1} \cup \dotsb \cup U_{v,\ell}$.
Hence the union $U_{v,1} \cup \dotsb \cup U_{v,\ell}$ is equal to $(x^\van)^{-1}(D(0,r_v,K_v))$.

\smallskip

Since the open subpaces $U_{v,k}$ constructed in this proof are always contained in the corresponding open subspaces~$U_{v,k}$ obtained from \cref{rigid-coordinate-disc}, they are simple.
\end{proof}

\bibliographystyle{amsalpha}
\bibliography{ZP}

\end{document}